\documentclass[ final, 11pt]{article}
\usepackage{hyperref}
\usepackage{amscd}
\usepackage{amsfonts}
\usepackage{indentfirst}
\usepackage{verbatim}
\usepackage{amsmath}
\usepackage{amsthm}
\usepackage{enumerate}
\usepackage{graphicx}
\usepackage{subfig}
\usepackage{color}
\usepackage[OT1]{fontenc}
\usepackage[latin1]{inputenc}
\usepackage[english]{babel}
\usepackage{amssymb}

\newtheorem{theorem}{Theorem}[section]
\newtheorem{lemma}{Lemma}[section]

\newtheorem{algorithm}{Algorithm}[section]
\newtheorem{corollary}{Corollary}[section]
\newtheorem{remark}{Remark}[section]
\newtheorem{assumption}{Assumption}[section]
\numberwithin{equation}{section}
\newcommand{\x}{{\bf x}}

\input{tcilatex}
\begin{document}

\title{A globally convergent numerical method for a 3D coefficient inverse
problem with a single measurement of multi-frequency data}
\author{Michael V. Klibanov\thanks{
Department of Mathematics and Statistics, University of North Carolina at
Charlotte, Charlotte, NC 28223, USA; (\texttt{mklibanv@uncc.edu}, \texttt{%
dnguye70@uncc.edu}, \texttt{lnguye50@uncc.edu}, \texttt{hliu34@uncc.edu}) } 
\thanks{%
corresponding author} \and Dinh-Liem Nguyen\footnotemark[1] \and Loc H.
Nguyen\footnotemark[1] \and Hui Liu\footnotemark[1] }
\date{}
\maketitle

\begin{abstract}
The goal of this paper is to reconstruct spatially distributed dielectric
constants from complex-valued scattered wave field by solving a 3D
coefficient inverse problem for the Helmholtz equation at multi-frequencies.
The data are generated by only a single direction of the incident plane
wave. To solve this inverse problem, a globally convergent algorithm is
analytically developed. We prove that this algorithm provides a good
approximation for the exact coefficient without any \textit{a priori}
knowledge of any point in a small neighborhood of that coefficient. This is
the main advantage of our method, compared with classical approaches using
optimization schemes. Numerical results are presented for both
computationally simulated data and experimental data. Potential applications
of this problem are in detection and identification of explosive-like
targets.
\end{abstract}

%
%
%

%



\textbf{Keywords.} global convergence, coefficient inverse problem, inverse
medium problem, coefficient identification, multi-frequency data,
experimental data, single measurement

\bigskip

\textbf{AMS subject classification.} 35R30, 78A46, 65C20

\section{Introduction}

\label{sec:1}


We are interested in this paper in the inverse problem of recovering the
spatially distributed dielectric constant of the Helmholtz equation from a
single boundary measurement of its solution. This inverse problem is also
called a coefficient inverse problem with a single measurement data (CIP).
CIPs are both ill-posed and highly nonlinear. Therefore, an important and
challenging question to address in a numerical treatment of a CIP is: \emph{%
How to obtain at least one point in a sufficiently small neighborhood of the
exact coefficient without any advanced knowledge of this neighborhood?} As
soon as this point is obtained, one can refine it using one of well
established techniques, such as, e.g., a Newton-like method or a
gradient-like method, see, e.g., \cite{Bak}. It is well known that a
rigorous guarantee of convergence of such a technique to the exact solution
can be obtained only if the starting point of iterations belongs to a
sufficiently small neighborhood of that solution.

We call a numerical method for a CIP \emph{approximately globally convergent 
}(globally convergent in short, or GCM) if: (1) a theorem is proved, which
claims that, under a reasonable mathematical assumption, this method
delivers at least one point in a sufficiently small neighborhood of the
exact solution without any advanced knowledge of this neighborhood, (2) the
error of the approximation of the true solution depends only on the level of
noise in the boundary data and on some discretization errors. Such an
assumption is necessary due to the well known \emph{tremendous} challenge of
the goal to develop those numerical methods for CIPs, which positively
address the question posed in the previous paragraph. This is especially
true for CIPs with single measurement data, which are known to be one of the
most challenging ones. In our particular case that assumption amounts
dropping a small term of an asymptotic expansion, and this is used only on
the zero iteration of our method, see Section \ref{sect:4.3}. Due to our
target application, the CIP considered here is a CIP with single measurement
data. We refer to \cite{BeilinaKlibanovBook,BK} for detailed discussions of
the notion of the approximate global convergence.

Our CIP has many applications in sonar imaging, geographical exploration,
medical imaging, near-field optical microscopy, nano-optics, see, e.g., \cite%
{ColtonKress:1998}. However, the target application of this paper is in
imaging of dielectric constants of antipersonnel mines and improvised
explosive devices (IEDs). We model the so-called \textquotedblleft stand
off" detection, which means that we use only backscatter data. Currently the
radar community relies only on the intensity of radar images \cite{Soum} to
see the geometrical information such as, e.g., sizes and shapes of targets.
Thus, we hope that the additional information about values of dielectric
constants of targets, which our method delivers, might lead in the future to
a better differentiation between explosives and clutter.

CIPs of wave propagation are a part of a bigger subfield, Inverse Scattering
Problems (ISPs). ISPs attract a significant attention of the scientific
community. There is a large body of the literature on imaging methods for
reconstructing geometrical information about targets, such as their shapes,
sizes and locations. We refer to, e.g.,~\cite{AmmariChowZou:sjap2016,
Ammar2004, Burge2005, Colto1996, Ito, Kirsc1998, LiLiuZou:smms2014,
LiLiuWang:jcp2014} and references therein for well-known imaging techniques
in inverse scattering such as level set methods, sampling methods, expansion
methods, and shape optimization methods.

%
%

The most popular approach to solve CIPs is nonlinear optimization methods
minimizing some cost functionals, see, e.g., \cite{Bak,Chavent,Engl,Gonch,T}%
. However, in general, such functionals are non-convex and might have
multiple local minima and ravines, see, e.g., Figure 1 in \cite{Scales} for
a convincing example of the latter. Therefore, a good initial guess for the
true solution is required. Such a first guess might be found by solving the
linearized problem using the Born approximation. As a results, the relative
target/background contrast is assumed to be small, while it is actually high
in many real world applications (see, e.g. the table of dielectric constants
on the web site of Honeywell with shortened link by Google
https://goo.gl/kAxtzB). This is one of the reasons for us to establish a GCM
to solve this inverse problem without knowing a good initial approximation
of the true solution.

In \cite{BeilinaKlibanovBook,BK} an approximately globally convergent method
was established for a CIP with single measurement data for a wave-like PDE
using Laplace transform. That method was verified on experimental data, see 
\cite{Klibanov:ip2010,Kuzh,TBKF1,TBKF2}, and references cited therein. The
GCM of \cite{BeilinaKlibanovBook,BK} was extended in \cite{Chow} to a CIP
for a different wave equation.

The goal of this paper is to develop a new GCM for a CIP for the Helmholtz
equation in $\mathbb{R}^{3}$ with a single measurement of multi-frequency
data. In other words, we now work in the frequency domain. Formally, the
Helmholtz equation can be obtained via the Fourier transform with respect to
time of the solution of the hyperbolic equation considered in \cite%
{BeilinaKlibanovBook,BK}. One of the reasons for us not to do so is because
the kernel of the Laplace transform, used in the previous GCM \cite%
{BeilinaKlibanovBook,BK}, decays exponentially fast and therefore some
significant information might be lost. The main difficulty in developing
this new GCM is to work with complex-valued functions where the maximum
principle, which plays an important role in the previous GCM, is no longer
applicable. We first develop the theory for the method. Next, we conduct a
numerical study for computationally simulated data. At the end, we present
one numerical example for experimental data. In fact, this is one of results
of the work \cite{Exp1} of our group.

In \cite{Exp1} the method of the current publication was successfully tested
on experimental data collected in the frequency domain. In addition, in \cite%
{Exp2} this method was tested on experimental data collected in the time
domain: after a certain preprocessing procedure, these data were ``moved" in
the frequency domain via the Fourier transform. In both works \cite%
{Exp2,Exp1}, results of the reconstructions of locations and dielectric
constants of inclusions were quite accurate. The 1D version of the method
under consideration along with its performance for experimental data can be
found in \cite{KlibanovLoc:ipi2016}.

As to the CIPs with multiple measurement, i.e., the Dirichlet-to-Neumann map
data, we mention the recent works \cite{Agal,Kab,Novikov} and references
cited therein, where reconstruction procedures are developed, which do not
require \textit{a priori} knowledge of a small neighborhood of the exact
coefficient.

The paper is organized as follows. In Section \ref{sec:2}, we state the
forward and inverse problems. In Section \ref{sec:3}, we discuss some facts
about the Lippmann-Schwinger equation. We need these facts for the analysis
of our numerical method. In Section \ref{sec:4}, we establish an integral
differential equation for the CIP and derive an initial approximation for
the target coefficient. In Section \ref{sec:5}, we present our globally
convergent numerical method and formulate the corresponding algorithm. In
Section \ref{sec:7}, we prove the main theorem of this paper: the global
convergence theorem. In addition, we discuss this theorem in Section \ref%
{sec:7} and present the main discrepancies between the theory and
computations. In Section \ref{sec:8}, we present our numerical results for
both computationally simulated and experimental data. Section \ref{sec:9} is
devoted to a short summary.

\section{Problem statement}

\label{sec:2}

\subsection{Assumptions on the coefficient $c( \mathbf{x}) $}

\label{sec:2.1}

Below $\mathbf{x}=(x_{1},x_{2},x_{3})\in \mathbb{R}^{3}.$ Let $B(R)=\{|%
\mathbf{x}|<R\}\subset \mathbb{R}^{3}$ be the ball of the radius $R>0$ with
the center at $\left\{ \mathbf{x}=0\right\} .$ Let $\Omega _{1}\Subset
\Omega \Subset B(R)$ be convex bounded domains and $\partial \Omega \in
C^{2+\alpha }$. Here and below $C^{m+\alpha }$ are Hölder spaces where $%
m\geq 0$ is an integer and $\alpha \in \left( 0,1\right) .$ Let $c(\mathbf{x}%
)$ be a function defined in $\mathbb{R}^{3}$. Throughout the paper we assume
that 
\begin{align}
c(\mathbf{x})& =1+\beta (\mathbf{x})\geq 1,\quad \mathbf{x}\in \mathbb{R}%
^{3},  \label{2.1} \\
\beta (\mathbf{x})& =0,\quad \mathbf{x}\in \mathbb{R}^{3}\setminus \Omega
_{1}.  \label{2.2}
\end{align}

The Riemannian metric generated by the function $c(\mathbf{x})$ is 
\begin{equation}
d\tau =\sqrt{c(\mathbf{x})}\left\vert d\mathbf{x}\right\vert ,\quad |d%
\mathbf{x}|=\sqrt{(dx_{1})^{2}+(dx_{2})^{2}+(dx_{3})^{2}}.  \label{2.9}
\end{equation}%
This metric generates geodesic lines. Let be $a>R$ an arbitrary number.
Consider the plane $P_{a}=\left\{ \mathbf{x}\in \mathbb{R}%
^{3}:x_{3}=-a\right\} .$ Note that since $c\left( \mathbf{x}\right) =1$
outside of the domain $\Omega _{1},$ then geodesic lines are straight lines
outside of $\Omega _{1}.$ In particular, since our incident plane wave
propagates along the positive direction of the $x_{3}-$axis (see this
section below), we consider for $\mathbf{x}\in \left\{ x_{3}<-R\right\} $
only those parts of geodesic lines which are parallel to the $x_{3}-$axis.
Then these lines intersect the plane $P_{a}$ orthogonally. Throughout the
paper we use the following assumption about the regularity of $c(\mathbf{x})$
and geodesic lines generated by the function $c(\mathbf{x}).$

\begin{assumption}
Assume that $c\in C^{15}(\mathbb{R}^{3})$. Furthermore, for any point $%
\mathbf{x}\in \mathbb{R}^{3}$ there exists unique geodesic line $L(\mathbf{x}%
)$ connecting $\mathbf{x}$ and the plane $P_{a}$ and such that $L\left( 
\mathbf{x}\right) $ intersects the plane $P_{a}$ orthogonally. \label%
{assumption c}
\end{assumption}

The length of the geodesic line $L\left( \mathbf{x}\right) $ is%
\begin{equation}
\tau \left( \mathbf{x}\right) =\dint\limits_{L\left( \mathbf{x}\right) }%
\sqrt{c\left( \zeta \right) }d\sigma .  \label{2.91}
\end{equation}%
A sufficient condition for the validity of Assumption was found in \cite%
{Romanov:ejmc2014}, 
\begin{equation*}
\sum_{i,j=1}^{3}\frac{\partial ^{2}\sqrt{c(\mathbf{x})}}{\partial
x_{i}\partial x_{j}}\xi _{i}\xi _{j}\geq 0,\quad \mbox{for }\text{ all }%
\mathbf{x},\xi =\left( \xi _{1},\xi _{2},\xi _{3}\right) \in \mathbb{R}^{3}.
\end{equation*}

\begin{remark}
\begin{enumerate}
\item The conditions in assumption \ref{assumption c} are technical ones.\
We use them only to derive the asymptotic expansion of the function $u\left( 
\mathbf{x},k\right) $ with respect to $k$ in Section \ref{sec:4.1}. This
expansion was proved in \cite{KlibanovRomanov:ip2016}, where these
conditions were used quite essentially. However, they are not imposed in
computations, see also Section \ref{sec:7.2}.


\item In the case of propagation of electromagnetic waves, as in the
experimental data of \cite{Exp2,Exp1}, $c\left( \mathbf{x}\right) $ is the
spatially distributed dielectric constant. We assume its independence on the
wavenumber $k$ since we end up working on a small $k-$interval. In this
paper, we assume that the  total  field satisfies the Helmholtz
equation rather than the full Maxwell's system. This is reasonable when 
the dielectric constant varies so slowly that it is almost constant over the
distances of the order of the wavelength, see \cite[Chapter 13]{BornWolf}
for more details. The good accuracy of our results for experimental data of 
\cite{Exp2,Exp1} as well as in Section \ref{sec:8.7} of this paper points
towards the validity of this assumption.
\end{enumerate}
\end{remark}

\subsection{The Coefficient Inverse Problem}

\label{sec:2.2}

Below $k>0$ is the wavenumber. Consider the function $u_{0}=\exp \left(
ikx_{3}\right) $, which represents the incident plane wave propagating along
the positive direction of the $x_{3}-$axis. Let $u(\mathbf{x},k)$ be the
solution of the following problem:%
\begin{equation}
\left\{ 
\begin{array}{l}
\Delta u+k^{2}c(\mathbf{x})u=0,\quad \mathbf{x}\in \mathbb{R}^{3}, \\ 
u=u_{0}+u_{\mathrm{sc}}, \\ 
\displaystyle\lim_{r\rightarrow \infty }r\left( \partial _{r}u_{\mathrm{sc}%
}-iku_{\mathrm{sc}}\right) =0,\quad r=|\mathbf{x}|\rightarrow \infty .%
\end{array}%
\right.  \label{2.3}
\end{equation}%
The functions $u,u_{0}$ and $u_{\mathrm{sc}}$ are called the total field,
incident field and scattered field respectively.

It is well-known (see \cite[Theorem 8.7]{ColtonKress:1998}) that problem (%
\ref{2.3}) has unique solution $u(\mathbf{x},k)\in C^{2}(\mathbb{R}^{3}).$
In fact, $u(\mathbf{x},k)\in C^{2+\alpha }(\mathbb{R}^{3})$ for any $\alpha
\in \left( 0,1\right) $, see \cite[Theorem 6.17]{GilbargTrudinger:1977}. In
this paper, we consider the following problem:

\textbf{Coefficient Inverse Problem (CIP)}. \emph{Let }$\underline{k}$\emph{%
\ and }$\overline{k}$\emph{\ be two positive numbers and }$\underline{k}<%
\overline{k}$\emph{. Assume that the function }$g(\mathbf{x},k),$\emph{\ } 
\begin{equation}
g(\mathbf{x},k)=u(\mathbf{x},k),\quad \mathbf{x}\in \partial \Omega ,k\in
\lbrack \underline{k},\overline{k}]  \label{2.4}
\end{equation}%
\emph{is known where }$u(\mathbf{x},k)$\emph{\ is the solution of the
problem \eqref{2.3}. Determine the function }$c(\mathbf{x})$\emph{\ for }$%
\mathbf{x}\in \Omega .$\emph{\ }

%

In the current paper, we develop the theory of our numerical method for this
CIP. We also numerically test our method for a CIP with partial data. Since
we have only a single direction of the propagation of the incident plane
wave, our CIP is an inverse problem with the data from a single measurement
event. In this paper, we focus only on the numerical method for our CIP. As
to the question of the uniqueness of the solution, it is worth mentioning
that currently uniqueness for the single measurement case can be proved only
if the homogeneous equation in (\ref{2.3}) would be replaced with a
non-homogeneous one, in which the right hand side would be nonzero
everywhere in $\overline{\Omega }$. The proof of this fact was first
introduced in \cite{BukhKlib} by using the Carleman estimate. There were
many works of many authors since then exploring the method of \cite{BukhKlib}%
. We refer the reader to review papers \cite{Klibanov:jiipp2013,
Yamamoto:ip2009}. In the case when the right hand side of equation in (\ref%
{2.3}) has some zeros, an attempt to address the uniqueness question faces
fundamental challenges which are yet unclear how to tackle. Hence, we assume
below uniqueness for our CIP.

It is well known in the theory of Ill-Posed Problems that one should assume
the existence of the exact solution, which corresponds to the ideal,
noiseless data \cite{Bak,BeilinaKlibanovBook,T}. Hence, throughout the
paper, $c^{\ast }(\mathbf{x}) $ is that exact coefficient and we use the
superscript $^{\ast }$ to indicate functions related to $c^{\ast }( \mathbf{x%
}) .$ Naturally, we assume that assumption \ref{assumption c} is valid for
the function $c^{\ast }( \mathbf{x}) .$

\section{Using the Lippmann-Schwinger equation}

\label{sec:3}

In this section we use the Lippmann-Schwinger equation to derive some
important facts, which we need both for our algorithm and for the
convergence analysis. We use here results of Chapter 8 of the book \cite%
{ColtonKress:1998}. Denote 
\begin{equation*}
\Phi (\mathbf{x},\mathbf{y})=\frac{\exp (i\overline{k}|\mathbf{x}-\mathbf{y}%
|)}{4\pi |\mathbf{x}-\mathbf{y}|},\text{ }\mathbf{x}\neq \mathbf{y}.
\end{equation*}%
In this section the function $\beta \in C^{\alpha }\left( \mathbb{R}%
^{3}\right) $ and it also satisfies conditions (\ref{2.1}), (\ref{2.2}). The
Lippmann-Schwinger equation for the function $u\left( \mathbf{x}\right) :=u(%
\mathbf{x},\overline{k})$ is%
\begin{equation}
u(\mathbf{x})=\exp (i\overline{k}x_{3})+\overline{k}^{2}\int_{\Omega }\Phi (%
\mathbf{x},\mathbf{y})\beta (\mathbf{y})u(\mathbf{y})d\mathbf{y}.
\label{8.3}
\end{equation}%
If the function $u\left( \mathbf{x}\right) $ satisfies equation (\ref{8.3})
for $\mathbf{x}\in \Omega ,$ then we can extend it for $\mathbf{x}\in 
\mathbb{R}^{3}\setminus \Omega $ via substitution these points $\mathbf{x}$
in the right hand side of (\ref{8.3}). Hence, to solve (\ref{8.3}), it is
sufficient to find the function $u\left( \mathbf{x}\right) $ only for points 
$\mathbf{x}\in \Omega .$ Consider the linear operator $K_{\beta }$ defined
as 
\begin{equation}
(K_{\beta }u)(\mathbf{x})=\overline{k}^{2}\int_{\Omega }\Phi (\mathbf{x},%
\mathbf{y})\beta (\mathbf{y})u(\mathbf{y})d\mathbf{y}.  \label{8.30}
\end{equation}%
It follows from Theorem 8.1 of \cite{ColtonKress:1998} that 
\begin{equation}
K_{\beta }:C^{\alpha }(\overline{\Omega })\rightarrow C^{2+\alpha }(\mathbb{R%
}^{3})\text{ and }\Vert K_{\beta }u\Vert _{C^{2+\alpha }(\mathbb{R}%
^{3})}\leq B\Vert \beta \Vert _{C^{\alpha }(\overline{\Omega })}\Vert u\Vert
_{C^{\alpha }(\overline{\Omega })}.  \label{8.4}
\end{equation}%
Here and below $B=B(\beta ,\overline{k},\Omega _{1},\Omega )>0$ denotes
different constants which depend only on listed parameters. Therefore, the
operator $K_{\beta }$ maps $C^{\alpha }(\overline{\Omega })$ to $C^{\alpha }(%
\overline{\Omega })$ as a compact operator. Hence, the Fredholm theory is
applicable to equation (\ref{8.3}). Lemmata \ref{lemma 3.1} and \ref{lemma
3.2} follow from Theorem 8.3 and Theorem 8.7 of \cite{ColtonKress:1998}
respectively.

\begin{lemma}
The function $u\in C^{2+\alpha }( \mathbb{R}^{3}) $ is the solution of the
problem (\ref{2.3}) at $k=\overline{k}$ if and only if $u$ is a solution of
equation (\ref{8.3}). \label{lemma 3.1}
\end{lemma}

\begin{lemma}
There exists unique solution $u\in C^{2+\alpha }( \mathbb{R}^{3}) $ of the
problem (\ref{2.3}). Consequently (Lemma \ref{lemma 3.1}), there exists
unique solution $u\in C^{2+\alpha }( \mathbb{R}^{3}) $ of the problem (\ref%
{8.3}) and these two solutions coincide. Furthermore, by the Fredholm theory
the following estimates hold%
\begin{align}
\Vert u\Vert _{C^{\alpha }(\overline{\Omega })} &\leq B\Vert \exp ( i%
\overline{k}x_{3}) \Vert _{C^{\alpha }(\overline{\Omega })},  \label{8.50} \\
\Vert u\Vert _{C^{2+\alpha }(\overline{\Omega })} &\leq B\Vert \exp ( i%
\overline{k}x_{3}) \Vert _{C^{2+\alpha }(\overline{\Omega })}.  \label{8.51}
\end{align}%
\label{lemma 3.2}
\end{lemma}

\begin{proof}
 We need to prove only estimate (\ref{8.51}). By (\ref{8.4})
and \eqref{8.50} 
\begin{equation}
\Vert K_{\beta }u\Vert _{C^{2+\alpha }(\overline{\Omega })}\leq B\Vert \beta
\Vert _{C^{\alpha }(\overline{\Omega })}\Vert u\Vert _{C^{\alpha }(\overline{%
\Omega })} \leq B\Vert \exp \left( i\overline{k}x_{3}\right) \Vert
_{C^{\alpha }(\overline{\Omega })}\leq B\Vert \exp \left( i\overline{k}%
x_{3}\right) \Vert _{C^{2+\alpha }(\overline{\Omega })}.  \label{8.40}
\end{equation}%
On the other hand, by (\ref{8.3}) 
\begin{equation}
\Vert u\Vert _{C^{2+\alpha }(\overline{\Omega })}\leq \Vert \exp \left( i%
\overline{k}x_{3}\right) \Vert _{C^{2+\alpha }(\overline{\Omega })}+\Vert
K_{\beta }u\Vert _{C^{2+\alpha }(\overline{\Omega })}.  \label{8.41}
\end{equation}%
Thus, estimate (\ref{8.51}) follows from (\ref{8.40}) and (\ref{8.41}). 
\end{proof}

Lemma \ref{lemma 3.3} follows from Lemmata \ref{lemma 3.1} and \ref{lemma
3.2} as well as from results of Chapter 9 of the book \cite%
{Vainberg:gbsp1989}.

\begin{lemma}
\label{lemma 3.3} For all $\mathbf{x}\in \overline{\Omega },k>0$ the
solution $u\left( \mathbf{x},k\right) $ of the problem (\ref{2.3}) is
infinitely many times differentiable with respect to $k$. Furthermore, $%
\partial _{k}^{n}u\in C^{2+\alpha }( \overline{\Omega }) $ and 
\begin{equation*}
\lim_{\epsilon \rightarrow 0}\| \partial _{k}^{n}u\left( \mathbf{x}%
,k+\epsilon \right) -\partial _{k}^{n}u\left( \mathbf{x},k\right)
\|_{C^{2+\alpha}(\overline \Omega) }=0,\emph{\ }n=0,1,...\emph{\ }
\end{equation*}
\end{lemma}

Let the function $\chi \in C^{2}(\mathbb{R}^{3})$ be such that%
\begin{equation}
\chi \left( \mathbf{x}\right) =\left\{ 
\begin{array}{ll}
1 & \text{ if }\mathbf{x}\in \Omega _{1}, \\ 
\text{between }0\text{ and }1 & \text{ if }\mathbf{x}\in \Omega \setminus
\Omega _{1}, \\ 
0 & \text{ if }\mathbf{x}\in \mathbb{R}^{3}\setminus \Omega .%
\end{array}%
\right.  \label{8.60}
\end{equation}%
The existence of such functions $\chi $ is well known from the Real Analysis
course. While Lemmata \ref{lemma 3.1} and \ref{lemma 3.2} are formulated for
the case when the function $\beta \left( x\right) $ is real valued, we now
consider the case of a complex valued analog of this function. We will need
this for our algorithm. Consider a complex valued function $\rho (\mathbf{x}%
)\in C^{\alpha }(\overline{\Omega })$. Let $\widehat{\rho }\left( \mathbf{x}%
\right) =\chi \left( \mathbf{x}\right) \rho \left( \mathbf{x}\right) $. Then 
\begin{equation}
\widehat{\rho }\in C^{\alpha }(\mathbb{R}^{3}),\quad \widehat{\rho }(\mathbf{%
x})=0\quad \text{ in }\mathbb{R}^{3}\setminus \Omega .  \label{8.7}
\end{equation}%
We have the lemma.

\begin{lemma}
\label{lemma 3.4} Let $\beta ^{\ast }(\mathbf{x})=c^{\ast }\left( \mathbf{x}%
\right) -1$. Let $u^{\ast }(\mathbf{x},\overline{k})$ be the solution of the
problem (\ref{2.3}) in which $c\left( \mathbf{x}\right) $ is replaced with $%
c^{\ast }\left( \mathbf{x}\right) $. Consider equation (\ref{8.3}) in which $%
\beta \left( \mathbf{x}\right) $ is replaced with $\widehat{\rho }\left( 
\mathbf{x}\right) ,$%
\begin{equation}
u_{\rho }(\mathbf{x},\overline{k})=\exp (i\overline{k}x_{3})+\overline{k}%
^{2}\int_{\Omega }\Phi \left( \mathbf{x},\mathbf{y}\right) \widehat{\rho }(%
\mathbf{y})u_{\rho }(\mathbf{y},\overline{k})d\mathbf{y},\quad \mathbf{x}\in
\Omega .  \label{8.8}
\end{equation}%
Then there exists a sufficiently small number $\theta ^{\ast }=\theta ^{\ast
}(\beta ^{\ast },\overline{k},\chi ,\Omega _{1},\Omega )\in \left(
0,1\right) $ depending only on listed parameters such that if $\Vert \rho
-\beta ^{\ast }\Vert _{C^{\alpha }(\overline{\Omega })}\leq \theta $ and $%
\theta \in (0,\theta ^{\ast })$, then equation (\ref{8.8}) has unique
solution $u_{\rho }\in C^{\alpha }(\overline{\Omega })$. Furthermore, the
function $u_{\rho }\in C^{2+\alpha }(\mathbb{R}^{3})$ and 
\begin{equation}
\Vert u_{\rho }(\mathbf{x},\overline{k})-u^{\ast }(\mathbf{x},\overline{k}%
)\Vert _{C^{2+\alpha }(\overline{\Omega })}\leq Z^{\ast }\theta ,
\label{8.9}
\end{equation}%
where the constant $Z^{\ast }=Z^{\ast }(\beta ^{\ast },\overline{k},\chi
,\Omega _{1},\Omega )>0$ depends only on the listed parameters.
\end{lemma}

\begin{proof} 
Below $Z^{\ast }$ denotes different positive constants
depending on the above parameters. We have $\beta ^{\ast }\left( \mathbf{x}%
\right) =\chi \left( \mathbf{x}\right) \beta ^{\ast }\left( \mathbf{x}%
\right) +\left( 1-\chi \left( \mathbf{x}\right) \right) \beta ^{\ast }\left( 
\mathbf{x}\right) $. Since, by (\ref{2.2}), the function $\beta ^{\ast
}\left( \mathbf{x}\right) =0$ outside of the domain $\Omega _{1}$, then (\ref%
{8.60}) implies that $(1-\chi \left( \mathbf{x}\right) )\beta ^{\ast }\left( 
\mathbf{x}\right) \equiv 0.$ Hence, $\beta ^{\ast }\left( \mathbf{x}\right)
=\chi (\mathbf{x})\beta ^{\ast }\left( \mathbf{x}\right) $ and $(\widehat{%
\rho }-\beta ^{\ast })(\mathbf{x})=\chi (\mathbf{x})(\rho -\beta ^{\ast })(%
\mathbf{x})$. Using notation (\ref{8.30}), we rewrite equation (\ref{8.8})
in the following equivalent form 
\begin{equation}
(I-K_{\beta ^{\ast }})u_{\rho }=\exp (i\overline{k}x_{3})+A_{\rho -\beta
^{\ast }}(u_{\rho }),\quad \mathbf{x}\in \Omega ,  \label{8.10}
\end{equation}%
where the linear operator $A_{\rho -\beta ^{\ast }}:C^{\alpha }\left( 
\overline{\Omega }\right) \rightarrow C^{\alpha }\left( \overline{\Omega }%
\right) $ is given by 
\begin{equation}
A_{\rho -\beta ^{\ast }}\left( u_{\rho }\right) \left( \mathbf{x}\right) =%
\overline{k}^{2}\int_{\Omega }\Phi \left( \mathbf{x},\mathbf{y}\right) \chi
\left( \mathbf{y}\right) (\rho -\beta ^{\ast })(\mathbf{y})u_{\rho }(\mathbf{%
y},\overline{k})d\mathbf{y},\quad \mathbf{x}\in \Omega .  \label{8.100}
\end{equation}%
We have for any function $p\in C^{\alpha }(\overline{\Omega })$ 
\begin{equation}
\Vert A_{\rho -\beta ^{\ast }}\left( p\right) \Vert _{C^{\alpha }(\overline{%
\Omega })}=\Vert \overline{k}^{2}\int_{\Omega }\Phi (\mathbf{x},\mathbf{y}%
)\chi (\mathbf{y})(\rho -\beta ^{\ast })(\mathbf{y})p\left( \mathbf{y}%
\right) d\mathbf{y}\Vert _{C^{2+\alpha }(\overline{\Omega })}\leq Z^{\ast
}\theta \Vert p\Vert _{C^{\alpha }(\overline{\Omega })}.  \label{8.101}
\end{equation}%
Therefore, 
\begin{equation}
\Vert A_{\rho -\beta ^{\ast }}\Vert \leq Z^{\ast }\theta .  \label{8.11}
\end{equation}%
It follows from Lemma \ref{lemma 3.2} and the Fredholm theory that the
operator $(I-K_{\beta ^{\ast }})$ has a bounded inverse operator $%
T=(I-K_{\beta ^{\ast }})^{-1}:C^{\alpha }(\overline{\Omega })\rightarrow
C^{\alpha }(\overline{\Omega })$ and 
\begin{equation}
\left\Vert T\right\Vert \leq Z^{\ast }.  \label{8.111}
\end{equation}%
Hence, using (\ref{8.10}), we obtain%
\begin{equation}
u_{\rho }=T(\exp (i\overline{k}x_{3}))+(TA_{\rho -\beta ^{\ast }})(u_{\rho
}).  \label{8.12}
\end{equation}%
It follows from (\ref{8.11}) and (\ref{8.111}) that there exists a
sufficiently small number $\theta ^{\ast }\in (0,1)$ depending on $\beta
^{\ast }$, $\overline{k},$ $\chi ,\Omega _{1},\Omega $ such that if $\theta
\in (0,\theta ^{\ast })$ and $\Vert \rho -\beta ^{\ast }\Vert _{C^{\alpha }(%
\overline{\Omega })}\leq \theta ,$ then the operator $(TA_{\rho -\beta
^{\ast }}):C^{\alpha }(\overline{\Omega })\rightarrow C^{\alpha }(\overline{%
\Omega })$ is a contraction mapping. This implies uniqueness and existence
of the solution $u_{\rho }\in C^{\alpha }(\overline{\Omega })$ of equation (%
\ref{8.12}), which is equivalent to equation (\ref{8.8}). Also, 
\begin{equation}
\Vert u_{\rho }\Vert _{C^{\alpha }(\overline{\Omega })}\leq Z^{\ast }\Vert
\exp (i\overline{k}x_{3})\Vert _{C^{\alpha }(\overline{\Omega })},
\label{8.13}
\end{equation}%
with a different constant $Z^{\ast }$. Furthermore, by (\ref{8.4}) the
function $u_{\rho }\in C^{2+\alpha }(\overline{\Omega }).$

We now prove estimate (\ref{8.9}). Let $\widetilde{u}(\mathbf{x})=u_{\rho }(%
\mathbf{x},\overline{k})-u^{\ast }(\mathbf{x},\overline{k})$. Since $%
(I-K_{\beta ^{\ast }})u^{\ast }=\exp (\mathrm{i}\overline{k}x_{3}),$ we
obtain the following analog of (\ref{8.10}) 
\begin{equation}
(I-K_{\beta ^{\ast }})\widetilde{u}=A_{\rho -\beta ^{\ast }}(u_{_{\rho
}}),\quad \mathbf{x}\in \Omega .  \label{8.14}
\end{equation}%
Hence, $\widetilde{u}=(TA_{\rho -\beta ^{\ast }})(u_{\rho })$. Hence, (\ref%
{8.11}) and (\ref{8.13}) lead to 
\begin{equation}
\Vert \widetilde{u}\Vert _{C^{\alpha }(\overline{\Omega })}\leq Z^{\ast
}\theta .  \label{8.15}
\end{equation}%
Next, we rewrite (\ref{8.14}) as%
\begin{equation}
\widetilde{u}=K_{\beta ^{\ast }}\widetilde{u}+A_{\rho -\beta ^{\ast
}}(u_{\rho }),\quad \mathbf{x}\in \Omega .  \label{8.16}
\end{equation}%
By (\ref{8.4}) and (\ref{8.100}) the right hand side of equation (\ref{8.16}%
) belongs to the space $C^{2+\alpha }(\overline{\Omega }).$ Hence, using (%
\ref{8.4}), (\ref{8.101}), (\ref{8.13}) and (\ref{8.15}), we obtain from (%
\ref{8.16}) that $\Vert \widetilde{u}\Vert _{C^{2+\alpha }(\overline{\Omega }%
)}\leq Z^{\ast }\theta .$ 
\end{proof}

\section{Integral differential equation formulation}

\label{sec:4}

In this section, we eliminate the target coefficient $c(\mathbf{x})$ from
the governing equation \eqref{2.3}. This process leads to an integral
differential equation. Solving this equation is equivalent to solving our
CIP. Note that this approach is different from locally convergent methods
with iterative optimization processes which typically rely on least-squares
formulation. Furthermore, under some mathematical assumptions, we rigorously
derive a good initial approximation for $c(\mathbf{x}).$ The analysis relies
on the high frequency asymptotic behavior of the total wave in the next
section.

\subsection{The asymptotics with respect to $k\rightarrow \infty $}

\label{sec:4.1}

Let $\tau (\mathbf{x})$ be the function defined in (\ref{2.91}). It was
proven in \cite[Theorem 1 and its consequence (4.26)]{KlibanovRomanov:ip2016}
that for $\mathbf{x}\in B(R)$ the asymptotic expansion of the function $u(%
\mathbf{x},k)$ with respect to $k$ is 
\begin{align}
& u(\mathbf{x},k)=A(\mathbf{x})\exp (ik\tau (\mathbf{x}))(1+O(1/k)),\quad
k\rightarrow \infty ,\mathbf{x}\in B(R),  \label{4.17} \\
& |O(1/k)|\leq \frac{B_{1}}{k},  \label{300}
\end{align}%
where the function $A(\mathbf{x})>0$ is defined in $\overline{B}(R)$ and the
number $B_{1}=B_{1}(B(R),c)>0$ depends only on listed parameters. By (\ref%
{4.17}) and (\ref{300}), given a coefficient $c(\mathbf{x})$, one can choose
a number $\kappa _{0}=\kappa _{0}(B(R),c)>0$ such that 
\begin{equation*}
u(\mathbf{x},k)\neq 0,\quad \text{ for all }\mathbf{x}\in \overline{B}%
(R),k\geq \kappa _{0}.
\end{equation*}%
We note that by (\ref{4.17}), (\ref{300}) we need only an estimate from the
below for the number $\kappa _{0}$ rather than its exact value. The
development of the theory of this paper would become much more complicated
if we would work with the numbers $\kappa _{0},\underline{k},\overline{k},$
which would depend on the function $c$. Hence, from now on, we assume that
these numbers are the same for all functions $c$ which we consider here.
This assumption is supported from the Physics standpoint by our study of
experimental data \cite{Exp2,Exp1}. Thus, below numbers $\overline{k},%
\underline{k},\kappa _{0},m_{0}$ are the same for all functions $c\left( 
\mathbf{x}\right) $ considered here and also 
\begin{align}
& \overline{k}>\underline{k}\geq \kappa _{0}>1,\quad k\in \lbrack \underline{%
k},\overline{k}],  \label{4.2} \\
& \min_{\overline{\Omega }}\min_{k\geq \kappa _{0}}|u(\mathbf{x},k)|\geq
m_{0}>0.  \label{4.200}
\end{align}

\subsection{The tail function and some auxiliary functions}

\label{sec:4.2}

We obtain from (\ref{4.2}), (\ref{4.200}) and a straightforward calculation
that the vector $\nabla u(\mathbf{x},\overline{k})/u(\mathbf{x},\overline{k}%
) $, $\mathbf{x}\in B(R)$, is $\func{curl}$ free. Hence, there exists a
function $V(\mathbf{x})$ such that 
\begin{equation}
\nabla V(\mathbf{x})=\frac{\nabla u(\mathbf{x},\overline{k})}{u(\mathbf{x},%
\overline{k})},\quad \mathbf{x}\in B(R).  \label{4.3}
\end{equation}%
This implies 
\begin{equation*}
\exp (-V(\mathbf{x}))\left( u(\mathbf{x},\overline{k})\nabla V(\mathbf{x}%
)-\nabla u(\mathbf{x},\overline{k})\right) =0,\quad \mathbf{x}\in B(R).
\end{equation*}%
Therefore, 
\begin{equation*}
\nabla \left( \exp (-V(\mathbf{x}))u(\mathbf{x},\overline{k})\right)
=0,\quad \mathbf{x}\in B(R).
\end{equation*}%
Hence, 
\begin{equation*}
u(\mathbf{x},\overline{k})=C\exp (V(\mathbf{x})),\quad \mathbf{x}\in B(R)
\end{equation*}%
where $C$ is a constant. The smoothness of the function $u(\mathbf{x},%
\overline{k})$ and (\ref{4.3}) imply that $V(\mathbf{x})\in C^{2+\alpha
}(B(R))$. Choosing $C=1$, we have found a function $V\in C^{2+\alpha }(B(R))$
such that 
\begin{equation}
u(\mathbf{x},\overline{k})=\exp (V(\mathbf{x})),\quad \mathbf{x}\in B(R).
\label{4.6}
\end{equation}%
%
%
%
%
%
%

\begin{remark}
\begin{enumerate}
\item The function $V(\mathbf{x})$, $\mathbf{x} \in B(R)$, is determined
uniquely up to the addition of a constant $2 \pi n i$, $n \in \mathbb{Z}$.
We can choose $n = 0$. This choice determines the function $V(\mathbf{x})$
uniquely. We name it the tail function.

\item The choice of $n$ does not effect the results of our reconstruction
method since we only use the derivatives of $V(\mathbf{x})$ (see Algorithm %
\ref{algorithm 1}) instead of $V(\mathbf{x}).$
\end{enumerate}
\end{remark}

Define the function $v(\mathbf{x},k),$ 
\begin{equation}
v(\mathbf{x},k)=-\int_{k}^{\overline{k}}\frac{\partial _{k}u(\mathbf{x}%
,\kappa )}{u(\mathbf{x},\kappa )}d\kappa +V(\mathbf{x}),\quad \mathbf{x}\in
B(R),k\in \lbrack \underline{k},\overline{k}].  \label{4.7}
\end{equation}

\begin{lemma}
For all $k\in \lbrack \underline{k},\overline{k}]$, the function $v(\mathbf{x%
},k)$ belongs to $C^{2+\alpha }(\overline{B( R)})$. Moreover 
\begin{align}
&\hspace*{1cm} u(\mathbf{x},k)=\exp (v(\mathbf{x},k)),  \label{4.8} \\
&\Delta v(\mathbf{x},k)+(\nabla v(\mathbf{x},k))^{2}=-k^{2}c(\mathbf{x}),
\label{4.9}
\end{align}
for all $\mathbf{x} \in B(R)$. \label{lemma 4.1}
\end{lemma}

\begin{proof} 
The fact that $v\in C^{2+\alpha }(\overline{B\left( R\right) 
})$ follows immediately from (\ref{4.2}), (\ref{4.200}), (\ref{4.7}) and
Lemma \ref{lemma 3.3}. Now, we prove only (\ref{4.8}) and (\ref{4.9}).
Differentiating (\ref{4.7}) with respect to $k$ and then multiplying both
sides of the resulting equation by $\exp (-v(\mathbf{x},k))$, we obtain 
\begin{equation*}
\partial _{k}\left( \exp (-v(\mathbf{x},k)u(\mathbf{x},k))\right) =0,\quad 
\mathbf{x}\in B(R).
\end{equation*}%
Hence, $u(\mathbf{x},k)={C}$ $\exp (v(\mathbf{x},k)),$ where ${C}={C}(%
\mathbf{x})$ is independent on $k$. In particular, letting $k=\overline{k}$
and using \eqref{4.6}, we have $C=1$. This implies \eqref{4.8}.

Substituting the function $u(\mathbf{x},k)$ from \eqref{4.8} in the
Helmholtz equation in \eqref{2.3}, we obtain \eqref{4.9}. 
\end{proof}


We next eliminate the function $c(\mathbf{x})$ from equation \eqref{4.9}. To
this end, define the function $q(\mathbf{x},k),$ 
\begin{equation}
q(\mathbf{x},k)=\partial _{k}v(\mathbf{x},k)=\frac{\partial _{k}u(\mathbf{x}%
,k)}{u(\mathbf{x},k)},\quad \mathbf{x}\in B(R).  \label{4.13}
\end{equation}%
Differentiating equation (\ref{4.9}) with respect to $k$, we obtain 
\begin{equation}
\Delta q(\mathbf{x},k)+2\nabla v(\mathbf{x},k)\nabla q(\mathbf{x},k)=\frac{1%
}{k}\left( \Delta v(\mathbf{x},k)+(\nabla v(\mathbf{x},k))^{2}\right) ,\quad 
\mathbf{x}\in B(R).  \label{4.14}
\end{equation}%
Plugging the function $v(\mathbf{x},k)$ of (\ref{4.7}) into (\ref{4.14}), we
obtain Lemma \ref{lemma q}. Keeping in mind our numerical method, it is
convenient to consider in this lemma that $\mathbf{x}\in \Omega $ rather
than the above $\mathbf{x}\in B\left( R\right) .$

\begin{lemma}[integral differential equation]
The function $q\left( \mathbf{x},k\right) $ satisfies the following
nonlinear integral differential equation 
\begin{multline}
k\Delta q(\mathbf{x},k)+2k\nabla \left( -\int_{k}^{\overline{k}}q(\mathbf{x}%
,\kappa )d\kappa +V(\mathbf{x})\right) \nabla q(\mathbf{x},k) \\
=2\Delta \left( -\int_{k}^{\overline{k}}q(\mathbf{x},\kappa )d\kappa +V(%
\mathbf{x})\right) +2\left[ \nabla \left( -\int_{k}^{\overline{k}}q(\mathbf{x%
},\kappa )d\kappa +V(\mathbf{x})\right) \right] ^{2}  \label{4.15}
\end{multline}%
for all $\mathbf{x}\in \Omega $ and $k\in \lbrack \underline{k},\overline{k}%
].$ Denoting $\psi \left( \mathbf{x},k\right) =\partial _{k}g\left( \mathbf{x%
},k\right) /g(\mathbf{x},k)$, we have the following boundary condition 
\begin{equation}
q\left( \mathbf{x},k\right) =\psi (\mathbf{x},k),\quad \mathbf{x}\in
\partial \Omega ,k\in \lbrack \underline{k},\overline{k}].  \label{4.16}
\end{equation}%
\label{lemma q}
\end{lemma}

Now, solving our CIP is equivalent to solving problem \eqref{4.14}--%
\eqref{4.16}. To do so, we iteratively approximate both functions $q(\mathbf{%
x}, k)$ and $V(\mathbf{x}),$ using the predictor-corrector scheme. Here
approximations for $V(\mathbf{x})$ are used as predictors and approximations
for $q(\mathbf{x}, k)$ are used as correctors. We first start from an
approximation $V_{0}(\mathbf{x})$ for the tail function $V(\mathbf{x})$.

\subsection{The initial approximation of the tail function}

\label{sect:4.3}

In this section we introduce a reasonable mathematical assumption mentioned
in Section \ref{sec:1}. The idea is to find an initial approximation for $V$
using the asymptotic behavior of $u(\mathbf{x},k)$ as $k\rightarrow \infty $%
. Let $c=c^{\ast }$ and let condition (\ref{4.2}) be fulfilled. We set in (%
\ref{4.17}) $u\left( \mathbf{x},k\right) =u^{\ast }\left( \mathbf{x}%
,k\right) $ and then ignore the term $O\left( 1/k\right) $ for $k\geq 
\overline{k}.$ It follows from (\ref{300}) that we can do this uniformly for
all $\mathbf{x}\in B(R).$ Hence, we obtain the following approximate formula%
\begin{equation}
u^{\ast }\left( \mathbf{x},k\right) =A^{\ast }\left( \mathbf{x}\right) \exp
(ik\tau ^{\ast }\left( \mathbf{x}\right) ),\quad \mathbf{x}\in B(R),k\geq 
\overline{k}.  \label{4.170}
\end{equation}%
Define the function $v^{\ast }(\mathbf{x},k)$ for $k\geq \overline{k}$ as in %
\eqref{4.7}$.$ Then, the result in Lemma \ref{lemma 4.1} is still valid for $%
k\geq \overline{k}.$ 
Since by (\ref{4.8}) $v^{\ast }\left( \mathbf{x},k\right) =\log u^{\ast }(%
\mathbf{x},k),$ we set, using (\ref{4.170}), 
\begin{equation}
v^{\ast }\left( \mathbf{x},k\right) =\ln A^{\ast }(\mathbf{x})+ik\tau ^{\ast
}(\mathbf{x}),\quad \mathbf{x}\in B(R),k\geq \overline{k}.  \label{4.18}
\end{equation}%
Since $\overline{k}$ is sufficiently large, we can drop the term $\ln
A^{\ast }(\mathbf{x})$ in (\ref{4.18}) and obtain 
\begin{equation}
v^{\ast }\left( \mathbf{x},k\right) =ik\tau ^{\ast }(\mathbf{x}),\quad 
\mathbf{x}\in B(R),k\geq \overline{k}.  \label{4.19}
\end{equation}%
Next, since by \eqref{4.7} $V^{\ast }\left( \mathbf{x}\right) =v^{\ast }(%
\mathbf{x},\overline{k}),$ then (\ref{4.19}) implies that 
\begin{equation}
V^{\ast }\left( \mathbf{x}\right) =i\overline{k}\tau ^{\ast }(\mathbf{x}).
\label{500}
\end{equation}%
Next, since $q\left( \mathbf{x},k\right) =\partial _{k}v(\mathbf{x},k),$ we
obtain from (\ref{4.19}) 
\begin{equation}
q^{\ast }(\mathbf{x},k)=i\tau ^{\ast }(\mathbf{x}),\quad \mathbf{x}\in
B(R),k\geq \overline{k}.  \label{4.190}
\end{equation}%
Substituting functions $V^{\ast }( \mathbf{x}) $  from (\ref{500}) and $%
q^{\ast }(\mathbf{x},\overline{k})$ from (\ref{4.190}) in problem (\ref{4.15}%
)--\eqref{4.16} at $k:=\overline{k}$, we obtain via a straightforward
calculation 
\begin{equation}
\left\{ 
\begin{array}{rcll}
\Delta \tau ^{\ast } & = & 0 & \text{in }\Omega , \\ 
\tau ^{\ast } & = & -i\psi ^{\ast }(\mathbf{x},\overline{k}) & \mbox{on }%
\partial \Omega .%
\end{array}%
\right.  \label{4.110}
\end{equation}%
Thus, we have obtained the Dirichlet boundary value problem (\ref{4.110})
for the Laplace equation with respect to the function $\tau ^{\ast }\left(
x\right) $. Recall that $\partial \Omega \in C^{2+\alpha }$ and that the
function $\psi ^{\ast }(x,\overline{k})\in C^{2+\alpha }(\partial \Omega ).$
By the Schauder Theorem \cite{Lad}, there exists unique solution $\tau
^{\ast }\in C^{2+\alpha }(\overline{\Omega })$ of the problem (\ref{4.110}).

In practice, however, we have the non-exact boundary data $\psi \left( 
\mathbf{x},k\right) $ rather than the exact data $\psi ^{\ast }(\mathbf{x}%
,k).$ Thus, we set the first approximation $V_{0}\left( \mathbf{x}\right) $
for the tail function $V\left( \mathbf{x}\right) $ as 
\begin{equation}
V_{0}\left( \mathbf{x}\right) =i\overline{k}\tau (\mathbf{x}),  \label{106}
\end{equation}%
where the function $\tau \left( \mathbf{x}\right) $ is the $C^{2+\alpha }(%
\overline{\Omega })$ solution of the following analog of problem (\ref{4.110}%
) 
\begin{equation}
\left\{ 
\begin{array}{rcll}
\Delta \tau & = & 0 & \text{in }\Omega , \\ 
\tau & = & -i\psi (\mathbf{x},\overline{k}) & \mbox{on }\partial \Omega .%
\end{array}%
\right.  \label{1060}
\end{equation}%
Using (\ref{4.9}), we define the first approximation $c_{0}\left( \mathbf{x}%
\right) $ for the exact coefficient $c^{\ast }\left( \mathbf{x}\right) $ as 
\begin{equation}
c_{0}\left( \mathbf{x}\right) =-\frac{1}{\overline{k}^{2}}\left( \Delta
V_{0}(\mathbf{x})+(\nabla V_{0}(\mathbf{x}))^{2}\right) .  \label{501}
\end{equation}

\textbf{\ }

\begin{lemma}
\label{lemma 4.3} Assume that relations (\ref{4.170})--(\ref{4.190}), (\ref%
{106}) and (\ref{1060}) are valid. Then there exists a constant $C=C\left(
\Omega \right) >0$ depending only on the domain $\Omega $ such that%
\begin{equation}
\Vert V_{0}-V^{\ast }\Vert _{C^{2+\alpha }( \overline{\Omega }) }\leq C%
\overline{k}\Vert \psi ( x,\overline{k}) -\psi ^{\ast }( x,\overline{k})
\Vert _{C^{2+\alpha }( \partial \Omega ) }.  \label{107}
\end{equation}
Consequently, 
\begin{equation}
\|c_0 - c^*\|_{C^{\alpha}(\overline \Omega)} \leq \frac{C}{\overline k}\Vert
\psi ( x,\overline{k}) -\psi ^{\ast }( x,\overline{k}) \Vert _{C^{2+\alpha
}( \partial \Omega ).}  \label{4.22}
\end{equation}
%
\end{lemma}

\begin{proof} 
Note that (\ref{4.110}) follows from (\ref{4.15}) and (\ref%
{4.170})--(\ref{4.190}). Denote $\widetilde{\tau }\left( x\right) =\tau
\left( x\right) -\tau ^{\ast }\left( x\right) .$ Then (\ref{4.110}) and (\ref%
{1060}) imply that 
\begin{equation*}
\left\{ 
\begin{array}{rcll}
\Delta \widetilde{\tau } & = & 0 & \text{in }\Omega , \\ 
\widetilde{\tau } & = & -i(\psi -\psi ^{\ast })(\mathbf{x},\overline{k}) & %
\mbox{on }\partial \Omega .%
\end{array}%
\right.
\end{equation*}
Hence, the Schauder Theorem \cite{Lad} leads to (\ref{107}). Using ( \ref%
{4.9}) and (\ref{501}), we obtain 
\begin{equation}
c_{0}(\mathbf{x})-c^{\ast }(\mathbf{x})=\frac{-1}{\overline{k}^{2}}(\Delta
V_{0}(\mathbf{x})-\Delta V^{\ast }(\mathbf{x})+(\nabla V_{0}(\mathbf{x}%
)-\nabla V^{\ast }(\mathbf{x}))(\nabla V_{0}(\mathbf{x})+\nabla V^{\ast }(%
\mathbf{x}))),  \label{502}
\end{equation}%
for all $\mathbf{x}$ in $\Omega .$ Estimate \eqref{4.22} follows from %
\eqref{107} and (\ref{502}). 
\end{proof}

\begin{remark}
\begin{enumerate}
\item Formulas (\ref{4.170})--(\ref{4.190}) form the first part of our
reasonable mathematical assumption mentioned in Section \ref{sec:1}. The
second part is that we assume that numbers $k_{0},\underline{k},\overline{k}$
are the same for all functions $c$ we consider here, see (\ref{4.2}), (\ref%
{4.200}) and the paragraph above them. We point out that formulas (\ref%
{4.170})--(\ref{4.190}) are used only on the first iteration of our method.
However, we do not use them on follow up iterations.

\item It follows from estimate \eqref{4.22} that the first approximation $%
c_{0}\left( \mathbf{x}\right) $ of the exact coefficient $c^{\ast }\left( 
\mathbf{x}\right) $ belongs to a small neighborhood of $c^{\ast }\left( 
\mathbf{x}\right) $, provided that the error in the data is sufficiently
small. Therefore, the global convergence property is achieved just at the
start of our iterative process.

%

\item We will prove in Theorem \ref{thm main} that the global convergence
property is still valid after certain further iterations. 
The latter property is important because it is well-known that the more data
are used, up to a certain point, the better reconstruction is obtained, also
see section \ref{sec:7.2}.

%
\end{enumerate}

\label{remark 4.2}
\end{remark}

\section{The algorithm}

\label{sec:5}

We first consider the discretization of problem \eqref{4.15}--\eqref{4.16}
with respect to the wavenumber $k$. Then, we prove the well-posedness of the
resulting problem. In the last part of this section, we propose the globally
convergent numerical algorithm.

\subsection{The discretization of the interval of wavenumbers}

\label{sec:5.1} To find $q$ and $V$ from problem~\eqref{4.15}--\eqref{4.16}
using an iterative process, we consider a discretization of this problem
with respect to the wavenumber $k$. We divide the interval $[\underline{k},%
\overline{k}]$ into $N$ subintervals with the uniform step size as follows 
\begin{equation}
\underline{k}=k_{N}<k_{N-1}<\dots <k_{1}<k_{0}=\overline{k},\quad
h=k_{n-1}-k_{n}=\frac{\overline{k}-\underline{k}}{N},\quad n=1,\dots ,N
\label{3}
\end{equation}%
We approximate the function $q\left( \mathbf{x},k\right) $ as a piecewise
constant function with respect to $k$, 
\begin{equation}
q_{0}(\mathbf{x})\equiv 0,\quad q_{n}(\mathbf{x}):=q(\mathbf{x},k_{n}),\quad
k\in \lbrack k_{n},k_{n-1}),n=1,\dots ,N.  \label{4.21}
\end{equation}%
Now, we define 
\begin{equation}
Q_{n}(\mathbf{x})=\sum_{j=0}^{n}q_{j}(\mathbf{x}).  \label{4.210}
\end{equation}%
Then 
\begin{equation*}
\int_{k_{n}}^{\overline{k}}q(\mathbf{x},\kappa )d\kappa =hq_{n}(\mathbf{x})+h%
{Q_{n-1}}(\mathbf{x}),\quad \mathbf{x}\in \Omega ,n=1,\dots ,N.
\end{equation*}%
Equation (\ref{4.15}) becomes%
\begin{multline}
k_{n}\Delta q_{n}+2k_{n}\nabla (-hq_{n}-hQ_{n-1}+V_{n-1})\nabla
q_{n}=2\Delta (-hq_{n}-h{Q_{n-1}}+V_{n-1}) \\
+2(\nabla (-hq_{n}-hQ_{n-1}+V_{n-1}))^{2},\quad \mathbf{x}\in \Omega .
\label{4.20}
\end{multline}%
We use here $V_{n-1}$ instead of $V$ since we will update the tail function
in our algorithm. By Lemma \ref{lemma 3.3} $q_{n}\left( \mathbf{x}\right)
=q_{n-1}\left( \mathbf{x}\right) +O(h),h\rightarrow 0.$ So, assuming that
the discretization step size $h$ is sufficiently small, we replace in (\ref%
{4.20}) $\nabla V_{n-1}\nabla q_{n}$ with $\nabla V_{n-1}\nabla q_{n-1}$.
Using again the fact that $h$ is sufficiently small, we also drop in (\ref%
{4.20}) the nonlinear term $-2hk_{n}\left( \nabla q_{n}\right) ^{2}$ as well
as the terms $-2\Delta (hq_{n}),-2\nabla (hq_{n}).$ However, we do not drop
terms $\nabla (hQ_{n-1}),\Delta (hQ_{n-1})$ since $Q_{n-1}$ is a sum of
other terms, see (\ref{4.210}). These approximations enable us to simplify
the analysis as well as the computations of our method. 
Using (\ref{4.20}), we obtain 
\begin{multline}
k_{n}\Delta q_{n}-2k_{n}h\nabla Q_{n-1}\nabla q_{n}= \\
-2k_{n}\nabla V_{n-1}\nabla q_{n-1}+2\Delta (-h{Q_{n-1}}+V_{n-1})+2(\nabla
(-hQ_{n-1}+V_{n-1}))^{2},\quad \mathbf{x}\in \Omega ,n=1,\dots ,N.
\label{4.29}
\end{multline}%
Let $\psi _{n}(\mathbf{x})=\psi \left( \mathbf{x},k_{n}\right) .$ Then (\ref%
{4.16}) and (\ref{4.21}) lead to the following Dirichlet boundary condition
for the function $q_{n}(\mathbf{x})$ 
\begin{equation}
q_{n}\left( \mathbf{x}\right) =\psi _{n}(\mathbf{x}),\quad \mathbf{x}\in
\partial \Omega .  \label{4.30}
\end{equation}

Thus, we have obtained the Dirichlet boundary value problem (\ref{4.29})--%
\eqref{4.30} for an elliptic equation for the function $q_{n}\left( \mathbf{x%
}\right) $. We now prove the solvability of this problem.

\subsection{The well-posedness of problem (\protect\ref{4.29})--(\protect\ref%
{4.30})}

\label{sec:6}

In this section we study the existence and uniqueness of the solution $%
q_{n}\in C^{2+\alpha }(\overline{\Omega })$ of the Dirichlet boundary value
problem (\ref{4.29})--(\ref{4.30}). In the case of real-valued functions,
existence and uniqueness of the solution follow from \cite[Theorem 6.14]%
{GilbargTrudinger:1977}. The proof is based on the maximum principle and
Schauder Theorem \cite{GilbargTrudinger:1977,Lad}. However, complex-valued
functions cause some additional difficulties, since the maximum principle
does not work in this case. Still, we can get our desired results for the
case of complex-valued functions using the assumption that the norm $%
\left\Vert h\nabla Q_{n-1}\right\Vert _{C^{1+\alpha }}$ is sufficiently
small. Keeping in mind the convergence analysis in the next section, it is
convenient to consider here problem (\ref{4.29})--(\ref{4.30}) in a more
general form 
\begin{equation}
\left\{ 
\begin{array}{rcll}
\Delta w-\nabla p\nabla w & = & f(\mathbf{x}) & \mbox{in }\Omega , \\ 
w & = & \mu (\mathbf{x}) & \mbox{on }\partial \Omega .%
\end{array}%
\right.  \label{5.1}
\end{equation}

%

\begin{lemma}
Assume that in (\ref{5.1}) all functions are complex valued ones and also
that $p\in C^{1+\alpha }(\overline{\Omega }),f\in C^{\alpha }(\overline{%
\Omega }),\mu \in C^{2+\alpha }(\partial \Omega ).$ Consider a number $%
\gamma \in \left( 0,1\right) .$ Then there exists a constant $%
C_{1}=C_{1}(\Omega )>0$ and a sufficiently small number $\sigma =\sigma
(C_{1},\gamma )\in \left( 0,1\right) $ such that if $C_{1}\sigma \leq \gamma 
$ and $\Vert \nabla p\Vert _{C^{\alpha }(\overline{\Omega })}\leq \sigma $,
then problem \eqref{5.1} has a unique solution $w\in C^{2+\alpha }(\overline{%
\Omega })$. Moreover, 
\begin{equation}
\left\Vert w\right\Vert _{C^{2+\alpha }(\overline{\Omega })}\leq \frac{C_{1}%
}{1-\gamma }\left( \left\Vert f\right\Vert _{C^{\alpha }(\overline{\Omega }%
)}+\left\Vert \mu \right\Vert _{C^{2+\alpha }\left( \partial \Omega \right)
}\right) .  \label{5.2}
\end{equation}%
\label{lem 5.1}
\end{lemma}

\begin{proof} Below in this paper $C_{1}=C_{1}\left( \Omega \right) >0$
denotes different positive constants depending only on the domain $\Omega .$
Let the complex-valued function $v\in C^{2+\alpha }(\overline{\Omega }).$
Consider the following Dirichlet boundary value problem with respect to the
function $U$:%
\begin{equation}
\left\{ 
\begin{array}{rcll}
\Delta U & = & \nabla p\nabla v+f(\mathbf{x}) & \mbox{in }\Omega , \\ 
U & = & \mu (\mathbf{x}) & \mbox{on }\partial \Omega .%
\end{array}%
\right.  \label{5.3}
\end{equation}%
The Schauder Theorem implies that there exists unique solution $U\in
C^{2+\alpha }(\overline{\Omega })$ of the problem (\ref{5.3}) and 
\begin{equation}
\Vert U\Vert _{C^{2+\alpha }(\overline{\Omega })}\leq C_{1}\left( \sigma
\Vert v\Vert _{C^{2+\alpha }(\overline{\Omega })}+\left\Vert f\right\Vert
_{C^{\alpha }(\overline{\Omega })}+\left\Vert \mu \right\Vert _{C^{2+\alpha
}\left( \partial \Omega \right) }\right) .  \label{5.4}
\end{equation}%
Hence, for each fixed pair $f\in C^{\alpha }(\overline{\Omega }),\mu \in
C^{2+\alpha }\left( \partial \Omega \right) $, we can define a map which
sends the function $v\in C^{2+\alpha }(\overline{\Omega })$ in the solution $%
U\in C^{2+\alpha }(\overline{\Omega })$ of the problem (\ref{5.3}), say $%
U=S_{f,\mu }\left( v\right) .$ Hence, $S_{f,\mu }:C^{2+\alpha }(\overline{%
\Omega })\rightarrow C^{2+\alpha }(\overline{\Omega }).$ Since the operator $%
S_{f,\mu }$ is affine and $C_{1}\sigma <1,$ (\ref{5.4}) implies that $%
S_{f,\mu }$ is a contraction mapping. Let the function $w=S_{f,\mu }\left(
w\right) $ be its unique fixed point. Then the function $w\in C^{2+\alpha }(%
\overline{\Omega })$ is the unique solution of the problem (\ref{5.1}). In
addition, by (\ref{5.4}) 
\begin{equation*}
\left\Vert w\right\Vert _{C^{2+\alpha }(\overline{\Omega })}\leq C_{1}\left(
\sigma \left\Vert w\right\Vert _{C^{2+\alpha }(\overline{\Omega }%
)}+\left\Vert f\right\Vert _{C^{\alpha }(\overline{\Omega })}+\left\Vert \mu
\right\Vert _{C^{2+\alpha }\left( \partial \Omega \right) }\right) ,
\end{equation*}%
which implies (\ref{5.2}). 
\end{proof}

The result bellow follows immediately from Lemma \ref{lem 5.1}, (\ref{4.29})
and (\ref{4.30}).

\begin{corollary}
Consider problem (\ref{4.29})--(\ref{4.30}) in which $V_{n-1}$ is replaced
with $V_{n,i-1}$ and $q_{n}$ is replaced with $q_{n,i}$, where $n=1,\dots ,N$
and $i=1,\dots ,m.$ Assume that 
\begin{equation*}
\nabla V_{n,i-1}\in C^{1+\alpha }(\overline{\Omega })\text{ and }q_{s}\in
C^{2+\alpha }(\overline{\Omega }),\quad s=1,\dots ,n-1.
\end{equation*}%
Suppose that $\Vert q_{s}\Vert _{C^{2+\alpha }(\overline{\Omega })}\leq Y$
where $Y>0$ is a constant. Let $\gamma \in \left( 0,1\right) $ and $%
C_{1}=C_{1}\left( \Omega \right) $ be the numbers of Lemma \ref{lem 5.1}.
Denote $a=\overline{k}-\underline{k}.$ Suppose that the number $a$ is so
small that 
\begin{equation*}
C_{1}Ya<\gamma .
\end{equation*}%
Then there exists a unique solution $q_{n,i}\in C^{2+\alpha }(\overline{%
\Omega })$ of the problem (\ref{4.29})--(\ref{4.30}). Moreover, 
\begin{equation*}
\Vert q_{n,i}\Vert _{C^{2+\alpha }(\overline{\Omega })}\leq \frac{C_{1}}{%
1-\gamma }\left( \Vert F_{n,i}\Vert _{C^{\alpha }(\overline{\Omega })}+\Vert
\psi _{n}\Vert _{C^{2+\alpha }(\partial \Omega )}\right) ,
\end{equation*}%
where 
\begin{equation*}
F_{n,i}=-2k_{n}\nabla V_{n,i-1}\nabla q_{n-1}+2\Delta (-h{Q_{n-1}}%
+V_{n,i-1})+2(\nabla (-hQ_{n-1}+V_{n,i-1}))^{2}.
\end{equation*}%
\label{col 5.1}
\end{corollary}

Now, we are ready to establish the globally convergent algorithm.

\subsection{The algorithm}

\label{sec:5.2}

\begin{algorithm}[Globally Convergent Algorithm for CIP]
~

\begin{enumerate}

\item Given $\nabla V_{0}$ (see Section \ref{sect:4.3}), set $Q_{0}=0.$

\item For $n=1,\dots ,N$

\begin{enumerate}

\item Set $q_{n, 0} = q_{n - 1},$ $\nabla V_{n,0}=\nabla V_{n-1}$.

\item \label{step inner} For $i=1,\dots ,m$ (for some $m\geq 1$)

\begin{enumerate}
\item \label{step qni} Find $q_{n,i}$ by solving the boundary value problem (%
\ref{4.29})--(\ref{4.30}), where $V_{n-1}$ is replaced with $V_{n,i-1}$. See
Corollary \ref{col 5.1} for the solvability of this problem.

\item \label{step cni} Update 
\begin{align}
\nabla v_{n,i}&=-h(\nabla q_{n,i} + \nabla Q_{n-1})+\nabla V_{n,i-1},
\label{4.33} \\
c_{n,i}&=\frac{1}{k_{n}^{2}}(\Delta v_{n,i}+(\nabla v_{n,i})^{2}).
\label{4.35}
\end{align}%
%
%
%
%
%
%
%
%
%

\item Find $u_{n,i}(\mathbf{x},\overline{k})$ by solving the
Lippmann-Schwinger equation (\ref{4.9}) in which the function $\hat{\rho}(%
\mathbf{y})$ is replaced with $\chi (\mathbf{y})(c_{n,i}-1)(\mathbf{y})$. 

\item Update $\nabla V_{n,i}(\mathbf{x})=\nabla u_{n,i}(\mathbf{x},
\overline k)/u_{n,i}(\mathbf{x}, \overline k).$
\end{enumerate}

\item Set $c_{n}=c_{n,m}$, $\nabla V_{n}=\nabla V_{n,m}$, $q_{n}=q_{n,m}$ and $Q_{n}=Q_{n-1}+q_{n}.$
\end{enumerate}

\item \label{step stop} Let $\overline{N}\in \{1,\dots ,N\}$ be the optimal
number for the stopping criterion. Set the function $c_{\overline{N}}$ as
the computed solution of CIP.
\end{enumerate}

\label{algorithm 1}
\end{algorithm}

\begin{remark}
The stopping criterion in step \ref{step stop} and the choice of $m$ in step %
\ref{step inner} in the algorithm above are addressed computationally. We
refer to Section \ref{sec:8.4} for a corresponding discussion. We also
specify the computation of (\ref{4.35}) in Section \ref{sec:8.3}.
\end{remark}


\section{Global convergence}

\label{sec:7}

In this section we formulate and prove the global convergence theorem, which
is the main theorem of this paper.\ Next, we present a discussion of this
result.

\subsection{The global convergence theorem}

\label{sec:7.1}

Let $\delta >0$ be the level of the error in the boundary data $\psi
_{n}\left( \mathbf{x}\right) $ in (\ref{4.30}) in the following sense 
\begin{equation}
\left\Vert \psi _{n}-\psi _{n}^{\ast }\right\Vert _{C^{2+\alpha }\left(
\partial \Omega \right) }\leq \delta <\delta +h.  \label{5.8}
\end{equation}

Following (\ref{4.13}), let $q^{\ast }\left( \mathbf{x},k\right) =\partial
_{k}u^{\ast }\left( \mathbf{x},k\right) /u^{\ast }\left( \mathbf{x},k\right)
.$ All approximations for the function $q^{\ast }\left( \mathbf{x},k\right) $
and associated functions with the accuracy $O\left( h\right) $ as $%
h\rightarrow 0,$ which are used in this section below, can be justified by
Lemma \ref{lemma 3.3}. For $n=1,\dots,N$, denote $q_{n}^{\ast }\left( 
\mathbf{x}\right) =q^{\ast }\left( \mathbf{x},k_{n}\right) .$ For $k\in %
\left[ k_{n},k_{n-1}\right) $ we obtain $q^{\ast }\left( \mathbf{x},k\right)
=q_{n}^{\ast }\left( \mathbf{x}\right) +O\left( h\right) $ as $h\rightarrow
0.$ Set $q_{0}^{\ast }\left( \mathbf{x}\right) \equiv 0.$ Similarly with (%
\ref{4.33}) define the gradient $\nabla v_{n}^{\ast }$ as 
\begin{equation}
\nabla v_{n}^{\ast }\left( \mathbf{x}\right) =-h\nabla {q_{n}^{\ast }}\left( 
\mathbf{x}\right) -h\nabla {Q_{n-1}^{\ast }}\left( \mathbf{x}\right) +\nabla
V^{\ast }\left( \mathbf{x}\right) ,\text{ }\mathbf{x}\in \Omega .
\label{5.9}
\end{equation}%
Since $\Delta v_{n}^{\ast }=\mbox{div}\left( \nabla v_{n}^{\ast }\right) ,$
by (\ref{4.9})%
\begin{equation}
c^{\ast }\left( \mathbf{x}\right) =-\frac{1}{k_{n}^{2}}\left( \Delta
v_{n}^{\ast }+\left( \nabla v_{n}^{\ast }\right) ^{2}\right) +F_{n}^{\ast
}\left( \mathbf{x}\right) ,  \label{5.10}
\end{equation}%
where the function $F_{n}^{\ast }\left( \mathbf{x}\right) $ is clarified in
this section below. While equation (\ref{4.20}) is precise, we have obtained
equation (\ref{4.29}) using some approximations whose error is $O\left(
h\right) $ as $h\rightarrow 0.$ This justifies the presence of the term $%
G_{n}^{\ast }\left( \mathbf{x}\right) $ in the following analog of the
Dirichlet boundary value problem (\ref{4.29})--(\ref{4.30}): 
\begin{align}
k_{n}\Delta q_{n}^{\ast }-2k_{n}\nabla hQ_{n-1}^{\ast }\nabla q_{n}^{\ast }&
=-2k_{n}\nabla V^{\ast }\nabla q_{n-1}^{\ast }+2\Delta (-hQ_{n-1}^{\ast
}+V^{\ast })  \notag \\
& \hspace*{2cm}+2(\nabla (-hQ_{n-1}^{\ast }+V^{\ast }))^{2}+G_{n}^{\ast }
\label{5.11} \\
q_{n}^{\ast }|_{\partial \Omega }& =\psi _{n}^{\ast },  \label{6.9}
\end{align}%
where $F_{n}^{\ast }\left( \mathbf{x}\right) $ and $G_{n}^{\ast }\left( 
\mathbf{x}\right) $ are error functions. Let $M>0$ be a generic constant
which we will specify later. We assume that 
\begin{equation}
\left\Vert q_{n}^{\ast }\right\Vert _{C^{2+\alpha }( \overline{\Omega })
}\leq M, \quad n=1,\dots,N.  \label{504}
\end{equation}

Now, we prove the main theorem of the paper. For brevity and also to
emphasize the main idea of the proof, we formulate and prove Theorem \ref%
{thm main} only for the case when inner iterations are absent, i.e. for the
case $m=1$. The case $m\geq 2$ is a little bit more technical and is,
therefore, more space consuming, while the idea is still the same. Thus, any
function $f_{nj}$ in above formulae should be $f_{n}$ below in this section.
Before stating the theorem, we recall that $V_{0}(\mathbf{x})$ is the first
approximation of the tail as in Section \ref{sec:4.2} and that $c_{n}$, $%
n\geq 1$, is computed as in Section \ref{sec:5.2}.

%
%
%

\begin{theorem}[global convergence]
Let $\chi \left( \mathbf{x}\right) $\textbf{\ }be the function in (\ref{8.60}%
) and $m_{0}$ be the number in (\ref{4.200}). Let inequalities (\ref{504})
be valid. Suppose that Assumption \ref{assumption c} as well as all
hypotheses of Corollary \ref{col 5.1} hold true  for $q_{n}:=q_{n}^{\ast }$
and $Y:=4M$. Assume further that the number $h+\delta \in (0,1)$ is
sufficiently small. Then there exists a constant $\mathcal{C}=\mathcal{C}%
(c^{\ast },\underline{k},\overline{k},\Omega ,\Omega _{1},\gamma ,\chi
,m_{0})>1$ depending only on listed paramaters such that 
\begin{equation}
\Vert c_{n}-c^{\ast }\Vert _{C^{\alpha }(\overline{\Omega })}\leq \mathcal{C}%
^{n}(h+\delta ),\quad n=1,\dots ,\overline{N},  \label{5.161}
\end{equation}%
where\textbf{\ }$\overline{N}\in \left[ 1,N\right] $ is the optimal number
for the stopping criterion as in Algorithm \ref{algorithm 1}. \label{thm
main}
\end{theorem}

\begin{remark}
We refer to Section \ref{sec:7.2} for a discussion of Theorem \ref{thm main}.
\end{remark}


\begin{proof}

For the convenience of the presentation, we introduce the error parameter $%
\eta ,$ 
\begin{equation}
\eta =h+\delta .  \label{5.7}
\end{equation}%
Hence, the error functions $F_{n}^{\ast }\left( \mathbf{x}\right) $ and $%
G_{n}^{\ast }\left( \mathbf{x}\right) $ in \eqref{5.10} and \eqref{5.11},
respectively, can be estimated as%
\begin{equation}
\Vert F_{n}^{\ast }\Vert _{C^{\alpha }(\overline{\Omega })}\leq M\eta ,\quad
\Vert G_{n}^{\ast }\Vert _{C^{\alpha }(\overline{\Omega })}\leq M\eta .
\label{5.12}
\end{equation}%
We also assume that for $n=1,\dots ,\overline{N},$ 
\begin{multline}
\max \left\{ \Vert \nabla V^{\ast }\Vert _{C^{\alpha }(\overline{\Omega }%
)},\Vert \Delta V^{\ast }\Vert _{C^{\alpha }(\overline{\Omega })},\Vert
q_{n}^{\ast }\Vert _{C^{2+\alpha }(\overline{\Omega })},\Vert \nabla
v_{n}^{\ast }\Vert _{C^{\alpha }(\overline{\Omega })},\Vert \Delta
v_{n}^{\ast }\Vert _{C^{\alpha }(\overline{\Omega })},\right. \\
\left. \Vert u^{\ast }(\mathbf{x},\overline{k})\Vert _{C^{2+\alpha }(%
\overline{\Omega })}\right\} \leq M.  \label{5.13}
\end{multline}

Denote $C_{2}=C_{2}\left( \Omega ,\gamma \right) =\max \left\{ C\left(
\Omega \right) ,C_{1}\left( \Omega \right) /\left( 1-\gamma \right) \right\}
>0,$ where $C\left( \Omega \right) $ and $C_{1}\left( \Omega \right) $ are
constants of Lemma \ref{lem 5.1} and Corollary \ref{col 5.1}, respectively.
Let $Z^{\ast }=Z^{\ast }\left( \beta ^{\ast },\overline{k},\chi ,\Omega
_{1},\Omega \right) >0$ be the number of Lemma \ref{lemma 3.4}, which
depends only on listed parameters. Assume that $M$ is so large that 
\begin{equation}
M>\max \left\{ 4,Z^{\ast },32/d,26C_{2},C_{2}\overline{k}\right\}
\label{5.131}
\end{equation}%
where 
\begin{equation}
d=\min [ m_{0}^{2},m_{0}^{4}] ,  \label{7.8}
\end{equation}%
Denote 
\begin{align}
D^{\gamma }\widetilde{V}_{n}& =D^{\gamma }V_{n}-D^{\gamma }V^{\ast }, & 
\widetilde{q}_{n}& =q_{n}-q_{n}^{\ast }, & \widetilde{v}_{n}&
=v_{n}-v_{n}^{\ast }, & \widetilde{Q}_{n-1}& =Q_{n-1}-Q_{n-1}^{\ast }, 
\notag  \label{5.14} \\
\widetilde{u}_{n}\left( \mathbf{x}\right) & =u_{n}\left( \mathbf{x},%
\overline{k}\right) -u^{\ast }\left( \mathbf{x},\overline{k}\right) , & 
\widetilde{c}_{n}& =c_{n}-c^{\ast }, & \widetilde{\psi }_{n}& =\psi
_{n}-\psi _{n}^{\ast }. & &
\end{align}%
%
%
%
%
Here $\gamma =\left( \gamma _{1},\gamma _{2},\gamma _{3}\right) $ is the
multi index with non-negative integer components and $\left\vert \gamma
\right\vert =\gamma _{1}+\gamma _{2}+\gamma _{3}.$ Using (\ref{107}), (\ref%
{5.8}), (\ref{5.7}) and (\ref{5.131}), we obtain 
\begin{equation}
\Vert \nabla \widetilde{V}_{0}\Vert _{C^{\alpha }(\overline{\Omega })},\Vert
\Delta \widetilde{V}_{0}\Vert _{C^{\alpha }(\overline{\Omega })}\leq C_{2}%
\overline{k}\eta \leq M\eta .  \label{5.140}
\end{equation}%
Hence, (\ref{5.13}) and (\ref{5.140}) imply that 
\begin{equation}
\Vert \nabla V_{0}\Vert _{C^{\alpha }(\overline{\Omega })}=\Vert \nabla 
\widetilde{V}_{0}+\nabla V^{\ast }\Vert _{C^{\alpha }(\overline{\Omega }%
)}\leq M\eta +M\leq 2M,\text{ }\Vert \Delta V_{0}\Vert _{C^{\alpha }(%
\overline{\Omega })}\leq 2M.  \label{5.1400}
\end{equation}%
Subtract equation (\ref{5.11}) from equation (\ref{4.29}). Also, subtract
the boundary condition in (\ref{6.9}) from the boundary condition in (\ref%
{4.30}). We obtain 
\begin{equation}
\left\{ 
\begin{array}{rcll}
\Delta \widetilde{q}_{n}-2h\nabla Q_{n-1}\nabla \widetilde{q}_{n} & = & 
\widetilde{P}_{n} & \mbox{in }\Omega , \\ 
\widetilde{q}_{n} & = & \widetilde{\psi }_{n} & \mbox{on }\partial \Omega ,%
\end{array}%
\right.  \label{5.15}
\end{equation}%
where 
\begin{multline}
\widetilde{P}_{n}=2h\nabla \widetilde{Q}_{n-1}\nabla q_{n}^{\ast }-2\left(
\nabla \widetilde{q}_{n-1}\nabla V_{n-1}+\nabla q_{n-1}^{\ast }\nabla 
\widetilde{V}_{n-1}\right)  \notag \\
+\frac{2(\nabla (-h(Q_{n-1}+Q_{n-1}^{\ast })+V_{n-1}+V^{\ast })\nabla (-h%
\widetilde{Q}_{n-1}+\widetilde{V}_{n-1}))}{k_{n}}  \label{5.150} \\
+\frac{\Delta (-h\widetilde{Q}_{n-1}+\widetilde{V}_{n-1})}{k_{n}}-\frac{%
G_{n}^{\ast }}{k_{n}}.  \notag
\end{multline}%
%
%
%

Let $n\geq 2$ and let $p_{n-1}>0$ be an integer. Assume by induction that 
\begin{equation}
\Vert \nabla \widetilde{V}_{n-1}\Vert _{C^{\alpha }(\overline{\Omega }%
)},\Vert \Delta \widetilde{V}_{n-1}\Vert _{C^{\alpha }\left( \overline{%
\Omega }\right) },\Vert \widetilde{q}_{s}\Vert _{C^{2+\alpha }\left( 
\overline{\Omega }\right) }\leq M^{p_{n-1}}\eta \leq M,  \label{5.18}
\end{equation}%
where $s=1,...,n-1.$ Note that in \eqref{5.18}, the former inequality is our
induction assumption while the latter inequality follows from (\ref{5.131})
and the assumption the $\eta $ is sufficiently small. We assume that 
\begin{equation}
p_{n-1}\in \lbrack 1,12\left( \overline{N}-1\right) -8].  \label{2}
\end{equation}%
Similarly with (\ref{5.1400}) we obtain from (\ref{504}) and (\ref{5.18})
that 
\begin{equation}
\Vert \nabla V_{n-1}\Vert _{C^{\alpha }(\overline{\Omega })},\Vert \Delta
V_{n-1}\Vert _{C^{\alpha }(\overline{\Omega })},\Vert q_{s}\Vert
_{C^{2+\alpha }\left( \overline{\Omega }\right) }\leq 2M.  \label{5.22}
\end{equation}%
It follows from (\ref{4.210}), (\ref{5.18}) and (\ref{5.22}) that 
\begin{align}
& \Vert h\widetilde{Q}_{n-1}\Vert _{C^{2+\alpha }(\overline{\Omega })}\leq
aM^{p_{n-1}}\eta ,\quad  \label{1} \\
& \Vert 2h{\nabla Q_{n-1}}\Vert _{C^{1+\alpha }(\overline{\Omega })}\leq 4Ma.
\label{5.220}
\end{align}%
%
%
%
Use the assumption for $a$ in Corollary \ref{col 5.1} with $Y=4M$, i.e., 
\begin{equation}
4C_{1}Ma<\gamma <1.  \label{5.16}
\end{equation}%
From (\ref{5.8}), (\ref{5.15}), (\ref{5.22}) (\ref{5.220}) and \eqref{5.16},
we obtain 
\begin{equation}
\Vert \widetilde{q}_{n}\Vert _{C^{2+\alpha }(\overline{\Omega })}\leq
C_{2}\Vert \widetilde{P}_{n}\Vert _{C^{\alpha }(\overline{\Omega }%
)}+C_{2}\eta .  \label{5.17}
\end{equation}%
We now want to express the number $p_{n}$ via $p_{n-1}.$ First, using (\ref%
{5.12}), (\ref{5.13}), (\ref{5.16}) and (\ref{5.15})-(\ref{1}), we estimate
the norm $\Vert \widetilde{P}_{n}\Vert _{C^{\alpha }(\overline{\Omega })}.$
In doing so, we keep in mind that by (\ref{4.2}) $k_{n}>1.$ We obtain 
\begin{align*}
\Vert \widetilde{P}_{n}\Vert _{C^{\alpha }(\overline{\Omega })}& \leq
2MaM^{p_{n-1}}\eta +6MM^{p_{n-1}}\eta +16MM^{p_{n-1}}\eta +MaM^{p_{n-1}}\eta
+M^{p_{n-1}}\eta +M\eta \\
& \leq 25MM^{p_{n-1}}\eta .
\end{align*}%
Hence, using (\ref{5.17}) we obtain $\Vert \widetilde{q}_{n}\Vert
_{C^{2+\alpha }(\overline{\Omega })}\leq 25C_{2}MM^{p_{n-1}}\eta +C_{2}\eta
\leq 26C_{2}MM^{p_{n-1}}\eta .$ Since by (\ref{5.131}) $M>26C_{2},$ then 
\begin{equation}
\Vert \widetilde{q}_{n}\Vert _{C^{2+\alpha }(\overline{\Omega })}\leq
M^{p_{n-1}+2}\eta .  \label{5.23}
\end{equation}%
Hence, 
\begin{equation}
\Vert q_{n}\Vert _{C^{2+\alpha }(\overline{\Omega })}\leq \Vert \widetilde{q}%
_{n}\Vert _{C^{2+\alpha }(\overline{\Omega })}+\Vert q_{n}^{\ast }\Vert
_{C^{2+\alpha }(\overline{\Omega })}\leq 2M.  \label{5.24}
\end{equation}%
Subtracting (\ref{5.9}) from (\ref{4.33}) and using (\ref{5.14}), we obtain 
\begin{equation*}
\nabla \widetilde{v}_{n}=-h\nabla \widetilde{q}_{n}-h\nabla \widetilde{Q}%
_{n-1}+\nabla \widetilde{V}_{n-1}.
\end{equation*}%
Hence, using (\ref{5.8}), (\ref{5.7}), (\ref{5.18}), (\ref{1}) and (\ref%
{5.23}), we obtain 
\begin{align}
\Vert \nabla \widetilde{v}_{n}\Vert _{C^{\alpha }(\overline{\Omega })},\Vert
\Delta \widetilde{v}_{n}\Vert _{C^{\alpha }(\overline{\Omega })}& \leq
M^{p_{n-1}+2}\eta ^{2}+MaM^{p_{n-1}}\eta +M^{p_{n-1}}\eta  \notag \\
& \leq 3M^{p_{n-1}}\eta \leq M^{p_{n-1}+1}\eta .  \label{5.25}
\end{align}%
Hence, by (\ref{5.13}) and (\ref{5.25}) 
\begin{equation}
\Vert \nabla v_{n}\Vert _{C^{\alpha }(\overline{\Omega })}\leq \Vert \nabla 
\widetilde{v}_{n}\Vert _{C^{\alpha }(\overline{\Omega })}+\Vert \nabla
v_{n}^{\ast }\Vert _{C^{\alpha }(\overline{\Omega })}\leq 2M.  \label{5.26}
\end{equation}%
Next, by (\ref{4.35}), (\ref{5.10}) and (\ref{5.14}) 
\begin{equation}
\widetilde{c}_{n}=-\frac{1}{k_{n}^{2}}\left( \Delta \widetilde{v}_{n}+\left(
\nabla v_{n}+\nabla v_{n}^{\ast }\right) \nabla \widetilde{v}_{n}\right)
-F_{n}^{\ast }.  \label{5.27}
\end{equation}%
In particular, since the right hand side of (\ref{5.27}) belongs to $%
C^{\alpha }(\overline{\Omega }),$ then the function $\widetilde{c}_{n}\in
C^{\alpha }(\overline{\Omega }).$ Recalling that $k_{n}^{2}\geq \underline{k}%
^{2}>1$ and using (\ref{5.12}), (\ref{5.24})--(\ref{5.27}), we obtain 
\begin{equation}
\Vert \widetilde{c}_{n}\Vert _{C^{\alpha }(\overline{\Omega })}\leq \left(
3M+1\right) M^{p_{n-1}+1}\eta +M\eta \leq 4MM^{p_{n-1}+1}\eta \leq
M^{p_{n-1}+3}\eta .  \label{5.28}
\end{equation}

%
%

We now estimate the norm $\Vert \widetilde{u}_{n}\Vert _{C^{2+\alpha }(%
\overline{\Omega })}$. Assume that $\eta$ is sufficiently small such that $%
M^{p_{n-1}+3}\eta \leq \theta ^{\ast },$ where $\theta ^{\ast }=\theta
^{\ast }\left( \beta ^{\ast },\overline{k},\chi ,\Omega _{1},\Omega \right)
>0$ was introduced in Lemma \ref{lemma 3.4}. Then it follows from Lemma \ref%
{lemma 3.4}, \eqref{5.131} and (\ref{5.28}) that 
\begin{equation}
\Vert \widetilde{u}_{n}\Vert _{C^{2+\alpha }(\overline{\Omega })}=\Vert
u_{n}(\mathbf{x},\overline{k})-u^{\ast }(\mathbf{x},\overline{k})\Vert
_{C^{2+\alpha }(\overline{\Omega })}\leq Z^{\ast }M^{p_{n-1}+3}\eta \leq
M^{p_{n-1}+4}\eta .  \label{5.29}
\end{equation}%
Hence, similarly to (\ref{5.26})%
\begin{equation}
\Vert u_{n}(\mathbf{x},\overline{k})\Vert _{2+C^{\alpha }(\overline{\Omega }%
)}\leq 2M.  \label{5.30}
\end{equation}%
Now we can estimate $\left\vert u_{n}(\mathbf{x},\overline{k})\right\vert $
from the below. Since the parameter $\eta $ is sufficiently small, then,
using (\ref{4.200}), (\ref{2}) and (\ref{5.29}), we obtain 
\begin{equation}
|u_{n}(\mathbf{x},\overline{k})|\geq |u^{\ast }(\mathbf{x},\overline{k})|-|%
\widetilde{u}_{n}\left( \mathbf{x}\right) |\geq m_{0}-M^{p_{n-1}+4}\eta \geq 
\frac{m_{0}}{2},\text{ }\forall \mathbf{x}\in \overline{\Omega }.
\label{5.32}
\end{equation}

We now are ready to estimate Hölder norms of $\nabla \widetilde{V}%
_{n},\Delta \widetilde{V}_{n},\nabla V_{n},\Delta V_{n}.$ It is obvious that
for any two complex valued functions $f_{1},f_{2}\in C^{\alpha }\left( 
\overline{\Omega }\right) $ such that $f_{2}\left( \mathbf{x}\right) \neq 0$
in $\overline{\Omega }$ the following estimate is valid 
\begin{equation}
\left\Vert \frac{f_{1}}{f_{2}}\right\Vert _{C^{\alpha }(\overline{\Omega }%
)}\leq \frac{\Vert f_{1}\Vert _{C^{\alpha }(\overline{\Omega })}\Vert
f_{2}\Vert _{C^{\alpha }(\overline{\Omega })}}{\left\vert f_{2}\right\vert
_{\min }^{2}},\quad \text{where }\>\left\vert f_{2}\right\vert _{\min
}=\min_{\overline{\Omega }}\left\vert f_{2}\right\vert .  \label{5.33}
\end{equation}%
Since $\nabla \widetilde{V}_{n}=\nabla u_{n}(\mathbf{x},\overline{k})/u_{n}(%
\mathbf{x},\overline{k})-\nabla u^{\ast }(\mathbf{x},\overline{k})/u^{\ast }(%
\mathbf{x},\overline{k}),$ we have 
\begin{equation}
\nabla \widetilde{V}_{n}=\frac{u^{\ast }(\mathbf{x},\overline{k})\nabla
(u_{n}(\mathbf{x},\overline{k})-u^{\ast }(\mathbf{x},\overline{k}))+(u^{\ast
}(\mathbf{x},\overline{k})-u_{n}(\mathbf{x},\overline{k}))\nabla u^{\ast }(%
\mathbf{x},\overline{k})}{u_{n}(\mathbf{x},\overline{k})u^{\ast }(\mathbf{x},%
\overline{k})}.  \label{5.34}
\end{equation}%
Hence, using (\ref{4.200}), (\ref{5.13}), (\ref{5.131}), (\ref{7.8}), (\ref%
{5.29}), (\ref{5.32}) and (\ref{5.33}), we obtain%
\begin{equation}
\Vert \nabla \widetilde{V}_{n}\Vert _{C^{\alpha }(\overline{\Omega })}\leq 
\frac{8}{d}M^{p_{n-1}+7}\eta \leq M^{p_{n-1}+8}\eta .  \label{5.35}
\end{equation}%
Next,%
\begin{equation}
\Delta \widetilde{V}_{n}=\left( \frac{\Delta u_{n}}{u_{n}}-\frac{\Delta
u^{\ast }}{u^{\ast }}\right) (\mathbf{x},\overline{k})-\left( \frac{\nabla
u_{n}}{u_{n}}-\frac{\nabla u^{\ast }}{u^{\ast }}\right) (\mathbf{x},%
\overline{k})\cdot \left( \frac{\nabla u_{n}}{u_{n}}+\frac{\nabla u^{\ast }}{%
u^{\ast }}\right) (\mathbf{x},\overline{k}).  \label{5.36}
\end{equation}%
We now estimate each term in (\ref{5.36}) separately. We obtain similarly
with (\ref{5.34}) and (\ref{5.35}) 
\begin{equation}
\left\Vert \frac{\Delta u_{n}\left( \mathbf{x},\overline{k}\right) }{%
u_{n}\left( \mathbf{x},\overline{k}\right) }-\frac{\Delta u^{\ast }\left( 
\mathbf{x},\overline{k}\right) }{u^{\ast }\left( \mathbf{x},\overline{k}%
\right) }\right\Vert _{C^{\alpha }(\overline{\Omega })}\leq
M^{p_{n-1}+8}\eta .  \label{5.37}
\end{equation}%
Next, using (\ref{7.8}), (\ref{5.13}), (\ref{5.30})-(\ref{5.33}), we obtain%
\begin{equation}
\left\Vert \frac{\nabla u_{n}\left( \mathbf{x},\overline{k}\right) }{%
u_{n}\left( \mathbf{x},\overline{k}\right) }+\frac{\nabla u^{\ast }\left( 
\mathbf{x},\overline{k}\right) }{u^{\ast }\left( \mathbf{x},\overline{k}%
\right) }\right\Vert _{C^{\alpha }(\overline{\Omega })}\leq \frac{32}{d}%
M^{2}.  \label{5.370}
\end{equation}%
Hence, using (\ref{5.131}) and (\ref{5.36})-(\ref{5.370}), we obtain%
\begin{equation}
\Vert \Delta \widetilde{V}_{n}\Vert _{C^{\alpha }(\overline{\Omega })}\leq
\left( \frac{32}{d}M^{2}+1\right) M^{p_{n-1}+8}\eta \leq
(M^{3}+1)M^{p_{n-1}+8}\eta \leq M^{p_{n-1}+12}\eta .  \label{5.38}
\end{equation}%
Similarly to (\ref{5.26}), we derive from (\ref{5.35}) and (\ref{5.38}) that 
$\Vert \nabla V_{n}\Vert _{C^{\alpha }(\overline{\Omega })},\Vert \Delta
V_{n}\Vert _{C^{\alpha }(\overline{\Omega })}\leq 2M.$ Summarizing, assuming
the validity of estimates (\ref{5.18}) and (\ref{2}), we have established
the following estimates:%
\begin{multline}
\max \left\{ \Vert \nabla \widetilde{V}_{n}\Vert _{C^{\alpha }(\overline{%
\Omega })},\Vert \Delta \widetilde{V}_{n}\Vert _{C^{\alpha }(\overline{%
\Omega })},\Vert \widetilde{q}_{n}\Vert _{C^{2+\alpha }(\overline{\Omega }%
)},\right.  \\
\left. \Vert \nabla \widetilde{v}_{n}\Vert _{C^{\alpha }(\overline{\Omega }%
)},\Vert \Delta \widetilde{v}_{n}\Vert _{C^{\alpha }(\overline{\Omega }%
)},\Vert \widetilde{c}_{n}\Vert _{C^{\alpha }(\overline{\Omega })}\right\}
\leq M^{p_{n-1}+12}\eta .  \label{5.39}
\end{multline}%
Hence, it follows from (\ref{5.18}) and (\ref{5.39}) that $p_{n}=p_{n-1}+12.$
Hence, $p_{n}=p_{1}+12\left( n-1\right) .$ We now need to find $p_{1}.$ Let $%
n=1$. Then (\ref{5.12}), (\ref{5.131}), (\ref{5.140})-(\ref{5.15}) and (\ref%
{5.17}) imply that 
\begin{equation}
\Vert \widetilde{q}_{1}\Vert _{C^{2+\alpha }(\overline{\Omega })}\leq
C_{2}\Vert \widetilde{P}_{1}\Vert _{C^{\alpha }(\overline{\Omega }%
)}+C_{2}\eta \leq 2C_{2}\left( 6M+1\right) M\eta +M\eta \leq M^{3}\eta .
\label{5.40}
\end{equation}%
Hence, (\ref{5.140}), (\ref{5.18}) and (\ref{5.40}) imply that $p_{1}=3.$
Hence, $p_{n}=12n-8.$ Thus, estimates (\ref{5.39}) are valid for $%
p_{n-1}+12=12n-8.$ The target estimate (\ref{5.161}) of this theorem follows
immediately from the estimate for $\Vert \widetilde{c}_{n}\Vert _{C^{\alpha
}(\overline{\Omega })}$ in (\ref{5.39}) by noticing that $M^{12n-8}\leq M^{12n-8n/\overline{N}}$
and then choosing $\mathcal{C}=M^{12-8/\overline{N}}$.
\end{proof}

\subsection{Discussion of Theorem 6.1}

\label{sec:7.2}

It seems that estimate (\ref{5.161}) indicates that the function $c_{1}$
should provide the best approximation for the exact coefficient $c^{\ast }$.
However, our computational experience tells us that more iterations should
be performed. In fact, it is well-known that the use of more frequencies
(i.e. more data) provides more stable reconstruction results. In addition,
the convergence estimate (\ref{5.161}) ensures that functions $c_{n}$ are
located in a sufficiently small neighborhood of the exact coefficient $%
c^{\ast },$ as long as $h+\delta $ is sufficiently small and the number of
iterations $n$ does not exceed $\overline{N}$. At the same time, these
estimates do not guarantee that the distances between functions $c_{n}$ and
the function $c^{\ast }$ decrease when the iteration number $n$ increases
from $n=1$ to $n=\overline{N}.$ Due to this phenomenon, the stopping rule
for our algorithm should be chosen computationally.\ A similar phenomenon
for Landweber iterations for a linear ill-posed problem is observed in Lemma
6.2 on page 156 of the book \cite{Engl}. Note that a similar statement can
be found on page 157 of \cite{Engl}.


As it was mentioned in the previous paragraph, we have obtained functions $%
c_{n}$, $n=1,\dots ,\overline{N}$, in a small neighborhood of the true
coefficient $c^{\ast }$. This is achieved without any advanced knowledge of
a small neighborhood of $c^{\ast }$. By (\ref{5.7}) the approximation
accuracy depends only on the level of the error $\delta $ in the boundary
data and the discretization step size $h$ with respect to the wavenumber $k$%
. Hence, Theorem \ref{thm main} implies the global convergence of our
algorithm. Furthermore, the global convergence is established for one of the
most challenging types of CIPs: a CIP with the data resulting from a single
measurement event. Two latter features form the \emph{key advantage} of
Theorem \ref{thm main}. On the other hand, the global convergence property
is achieved within the framework of a reasonable mathematical approximation
indicated in Section \ref{sect:4.3}. Hence, to be more precise, this is the
approximate global convergence property, as defined in Introduction and in 
\cite{BeilinaKlibanovBook,BK}.


It can be seen both from our algorithm and from the proof of Theorem \ref%
{thm main} that the approximations (\ref{4.170})-(\ref{4.190}) are used only
to obtain the first approximation $V_{0}$ for the tail function and are not
used on follow up iterations with $n=1,\dots,\overline N$. Recall that the
number of iterations $\overline{N}$ can be considered sometimes as a
regularization parameter in the theory of ill-posed problems \cite%
{BeilinaKlibanovBook, Engl, T}.

\subsection{Main discrepancies between the theory and numerical
implementation}

\label{sec:7.3}

In our opinion, it is hard to anticipate that some discrepancies between the
theory and computations would not occur for such a challenging problem as
the one considered in this paper. So, as it often happens with other
numerical methods, some conditions are relaxed in computations, compared
with ones in the convergence theorem. Indeed, it is well known that the
theory is usually more pessimistic than the computational results are. This
is because the theory cannot grasp all features of computations. The main
discrepancies are:

\begin{enumerate}
\item The Assumption \ref{assumption c} is used only to obtain the
asymptotic formula (\ref{4.17}), which was derived in \cite%
{KlibanovRomanov:ip2016}. However, since our target applications are in
imaging antipersonnel mines and IEDs, the function $c(\mathbf{x})$, in our
numerical testing in Section \ref{sec:8} as well as in \cite{Exp2,Exp1}, is
discontinuous on the inclusion/background interface.


\item Instead of computing the initial approximation of the tail as in
Section \ref{sect:4.3}, we find its derivatives directly by solving problem %
\eqref{1033}--\eqref{103}. This enables us to avoid the error arising from
numerical differentiation. We recall that our algorithm requires the
derivatives of the tail function rather than this function itself. In
addition, the right hand side of formula (\ref{4.35}) might be non-positive.
Hence, we amend this formula in computations via using formula (\ref{8.2})
and a smoothing procedure (Section 7.3).

\item Our theory works only for the case when the data are given on the
entire boundary of the domain $\Omega .$ However, due to our target
applications in detecting and identifying explosives, we are also interested
in the case of backscatter data. Thus, we complement the backscatter data as
in (\ref{completion}).

\item We use the so-called \textquotedblleft data propagation procedure" in
Section \ref{sec:8.2}. This is a heuristic procedure, which, nevertheless,
is very effective and it is widely used in the optics community, see, e.g., 
\cite[Chapter 2]{Novot2012}. In particular, this procedure helps us to
reduce the search domain as in (\ref{8.1}). We also refer to~\cite{Buhan2013}
for a time-reversal technique that was exploited for reducing the
computational domain of a coefficient inverse problem.
\end{enumerate}

\section{Numerical study}

\label{sec:8}

We describe in this section some details of the numerical implementation of
the above globally convergent algorithm. Also, we present some
reconstruction results for both computationally simulated and experimental
data. It is convenient to denote in this section points $\mathbf{x}=\left(
x_{1},x_{2},x_{3}\right) $ as $\mathbf{x=}\left( x,y,z\right) .$

\subsection{Dimensionless variables}

\label{sec:8.1}

We want to demonstrate in our computations that our algorithm works with
ranges of parameters, which are realistic in our desired application in
imaging of antipersonnel mines and IEDs. Hence, we start from variables,
which have dimensions and make them dimensionless then. We are guided by our
experimental data in \cite{Exp1}.

Let 1 $cm$ = 1 centimeter. Consider the dimensionless variable $\mathbf{x}=%
\mathbf{x/}\left( 10\ cm\right) .$ We do not change notations for brevity.
For example, the dimensionless 0.6 means 6 cm. The experimental data of \cite%
{Exp1} are stable only in a small neighborhood of $3.1$ GHz. Motivated by
this fact, we will work on a corresponding interval of wavenumbers $%
[6.2,6.7] $ in our numerical study. Following the measurement arrangement of 
\cite{Exp1}, we assume in our numerical study that the backscatter data are
measured on a $100cm\times 100cm$ rectangle, which we call \textquotedblleft
measurement plane". We assume that this plane is 76 cm away from the front
face of the target to be imaged. Hence, in dimensionless variables we have
measurements on a $10\times 10$ rectangle.

\subsection{Data propagation}

\label{sec:8.2}

The data propagation process is applied to the function $u_{sc}(\mathbf{x}%
,k) $ on the measurement plane and it aims to approximate the data for $%
u_{sc}(\mathbf{x},k)$ on a plane, which is closer to the target. In other
words, we ``move" the data closer to the target. In doing so, we assume that
we know in advance some rough estimates of locations of targets. However, we
do not assume that we know exact locations. Such an advanced knowledge of
those estimates goes along well with the main principles of the
regularization theory \cite{Bak,BeilinaKlibanovBook,Engl,T}.

The data propagation process firstly reduces the computational domain for
our algorithm and secondly makes the scattered field data look more focused.
The latter feature also provides us a reasonable estimate for the $xy-$%
location of the target, which is important for the numerical implementation.
The data propagation is done using the so called angular spectrum
representation, which is a well known technique in optics, see, e.g., \cite[%
Chapter 2]{Novot2012}. We also refer to~\cite{Exp2,Exp1,TBKF2}, where this
process has been exploited for the preprocessing of the experimental data.
We display in Figure~\ref{fig 1a} and Figure~\ref{fig 1b} the absolute value
of the noisy scattered field $u_{sc}(\mathbf{x},k)$ before and after
propagation, respectively. The scatterer in this case consists of two
different inclusions as in Figure~\ref{fi:4}. From the propagated data in
Figure~\ref{fig 1b} we can see two clear peaks at the locations of the
inclusions. Comparison of Figure \ref{fig 1a} with Figure \ref{fig 1b}
demonstrates the effectiveness of the data propagation procedure.



\begin{figure}[h!]
\centering
\subfloat[\label{fig 1a} Magnitude of the backscatter
data with 15 \% noise]{\includegraphics[width=6.5cm]{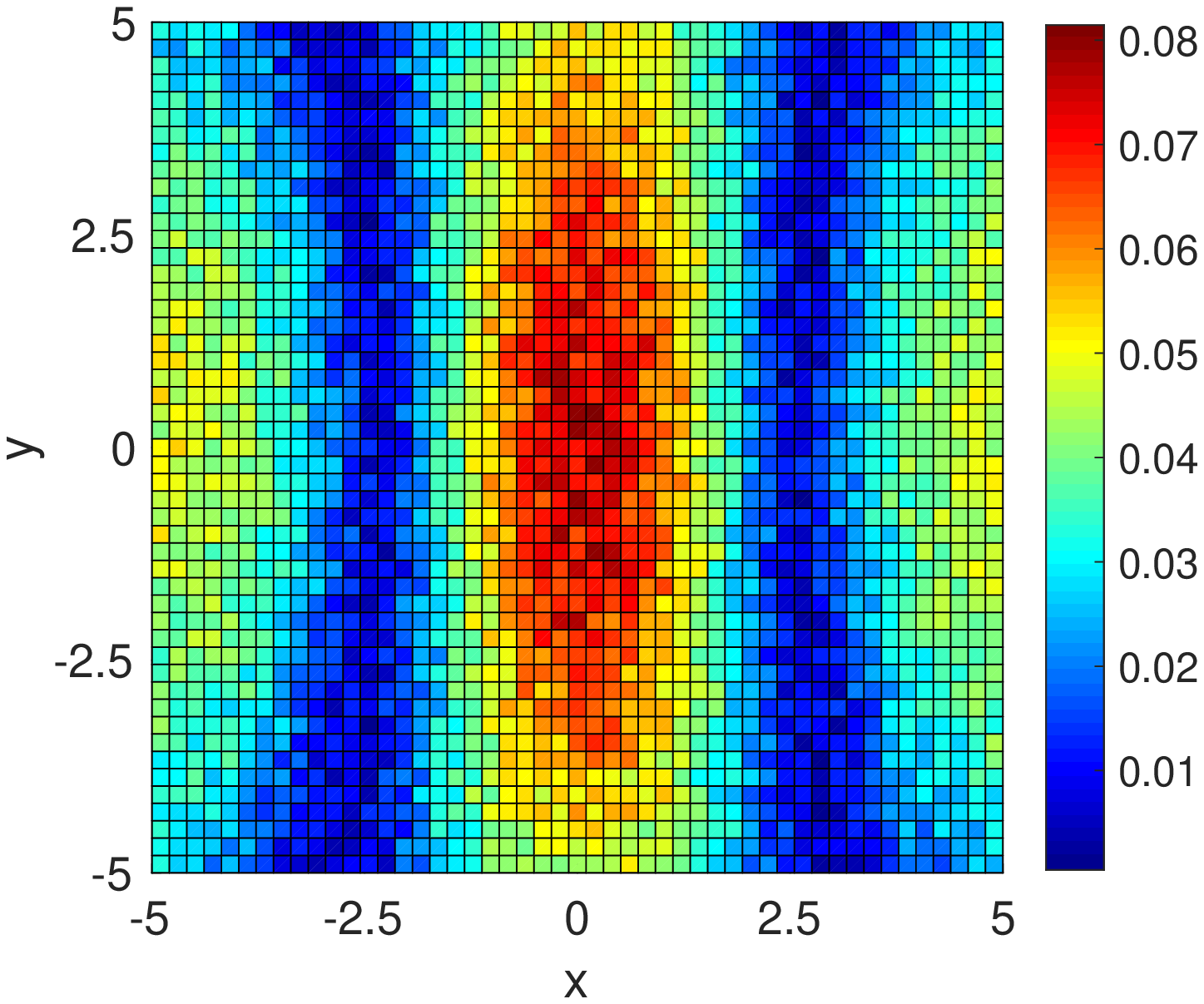}} 
\hspace{0.5cm} 
\subfloat[\label{fig 1b}Magnitude of the propagated
data]{\includegraphics[width=6.5cm]{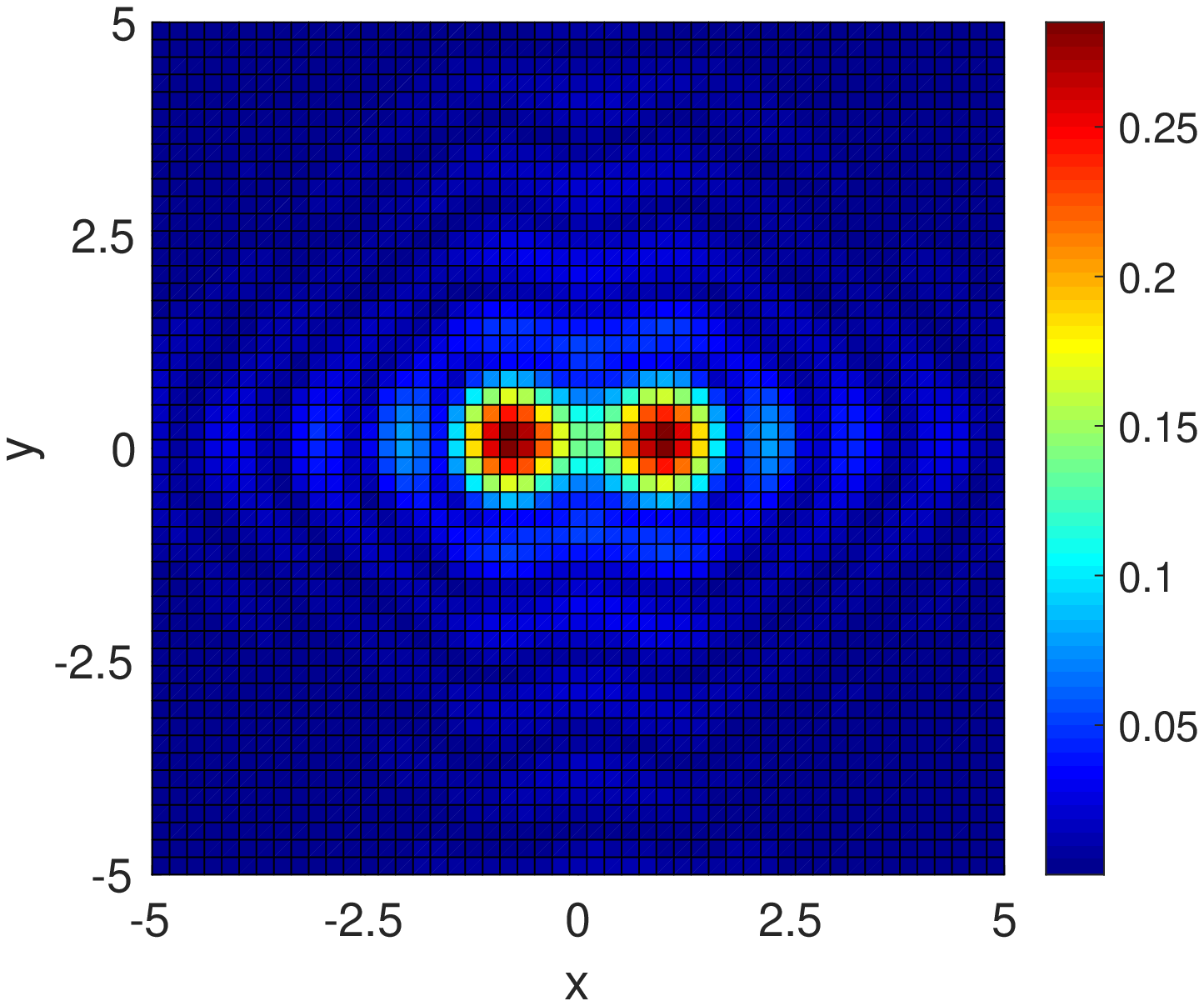}}
\caption{For $k = 6.48$, we present in (a) the absolute value of the noisy
backscattered field on the rectangle $(-5,5)^2\times \{z=-7.6\}$ and in (b)
the absolute value of the propagated data on the rectangle $(-5,5)^2\times
\{z=-0.75 \}$.}
\label{fi:1}
\end{figure}

\subsection{Some computational details}

\label{sec:8.3}


%
%
%
%
%

In our computations, $[\underline{k},\overline{k}]=[6.2,6.7]$ and $N=9$.
This means that we have 10 wavenumbers in this interval and from now on we
refer to them as $k_{n}$ for $n=0,\dots ,9$. All computationally simulated
data considered below are perturbed by 15$\%$ of random artificial noise.
More precisely, for each $k_{n}$ and for each grid point $\mathbf{x}$ on the
measurement plane, the total field data were perturbed as 
\begin{equation}
g_{noise}(\mathbf{x},k_{n})=g(\mathbf{x},k_{n})+0.15\Vert g(\mathbf{x}%
,k_{n})\Vert _{L^{2}}\frac{\sigma (\mathbf{x},k_{n})}{\Vert \sigma (\mathbf{x%
},k_{n})\Vert _{L^{2}}},  \label{4}
\end{equation}%
where $\sigma (\mathbf{x},k_{n})=\sigma _{1}(\mathbf{x},k_{n})+i\sigma _{2}(%
\mathbf{x},k_{n})$, and $\sigma _{1}(\mathbf{x},k_{n})$ and $\sigma _{2}(%
\mathbf{x},k_{n})$ are random numbers, uniformly distributed on $(-1,1)$.
The $L^{2}$ norm is taken over the measurement plane. Now assume that $f(%
\mathbf{x},k)$ is the data obtained after the propagation. More precisely,
let $f(\mathbf{x},k)$ be the sum of the propagated scattered field and the
incident field. In our algorithm we need the derivative of $f(\mathbf{x},k)$
with respect $k$. We compute it via the finite difference as 
\begin{equation*}
\frac{\partial f(\mathbf{x},k_{n})}{\partial k}=\frac{f(\mathbf{x},k_{n})-f(%
\mathbf{x},k_{n-1})}{h},\quad n=1,2,\dots ,9.
\end{equation*}%
Even though the differentiation of the noisy data is an ill-posed problem,
this approximation turned out to work well in our numerical implementation.
This is partly because the propagated data (Figure~\ref{fig 1b}) have less
noise than the given data (Figure \ref{fig 1a}). 

Each of our targets to be imaged has its front face at the plane $\left\{
z=0\right\} $. The backscatter data $u_{sc}\left( \mathbf{x},k\right) $ are
measured at the plane $\left\{ z=-7.6\right\} $ on the above mentioned
rectangle 
\begin{equation}
\left\{ \left( x,y,z\right) :\left( x,y\right) \in \left( -5,5\right) \times
\left( -5,5\right) ,z=-7.6\right\} .  \label{7}
\end{equation}%
Next, the backscatter data are propagated to the plane $\left\{
z=-0.75\right\} .$ We have noticed that the propagated function $u_{sc}$ is
close to zero outside of the rectangle 
\begin{equation}
\Gamma =\left\{ \left( x,y,z\right) :\left( x,y\right) \in \left(
-2.5,2.5\right) \times \left( -2.5,2.5\right) ,z=-0.75\right\} ,  \label{6}
\end{equation}%
see Figure~\ref{fi:1}(b). Hence, our computational domain is 
\begin{equation}
\Omega =\left\{ \left( x,y,z\right) \in (-2.5,2.5)\times (-2.5,2.5)\times
(-0.75,4.25)\right\}  \label{8.5}
\end{equation}%
and we treat $\Gamma =\overline{\Omega }\cap \left\{ z=-0.75\right\} $ as
the surface where we have the propagated data.

We have observed that the data propagation process also provides a very good
assistance in determining the locations of the targets of interest in the $%
xy-$plane. More precisely, the peaks that can be seen in magnitude of the
propagated data (Figure~\ref{fig 1b}), are located exactly at the location
of those targets in the $xy-$plane. Therefore, we can consider our search
domain in the $xy-$plane as 
\begin{equation}
\Omega _{T}=\cup _{j=1}^{J}\{(x,y,z):|f(\mathbf{x},k)|>0.7\max_{O_{j}}|f(%
\mathbf{x},k)|,z=-0.75\},  \label{8.1}
\end{equation}%
%
%
%
%
%
%
%
%
where $J$ is the number of peaks of $|f(\mathbf{x},k)|$ on the set $\Gamma $
and $O_{j}$ is a neighborhood of the point where $|f(\mathbf{x},k)|$ attains
its $j-$th peak. These neighborhoods should not intersect with each other.
The truncation value 0.7 was chosen by trial-and-error tests.

When running our algorithm, we do not assume the knowledge of the fact that
the front face of any target is located on the plane $\left\{ z=0\right\} .$
Still, as we have stated above, we assume the knowledge of some rough
estimates of linear sizes of targets of interest. In the $z-$direction we
search for the target in the range $(-0.75,1)$. This search range is
motivated by our desired applications for detection of explosives. Indeed, $%
z\in (-0.75,1)$ means that the linear size of the target in the $z-$%
direction does not exceed 17.5 cm while linear size of antipersonnel mines
and IEDs are typically between 5 cm and 15 cm. Note that similar ranges were
used in~\cite{Exp2, Exp1, TBKF2, TBKF1}. Thus, our search domain for the
target is $\Omega _{T}\times (-0.75,1)$. In addition, the right hand side of
the formula (\ref{4.35}) for the function $c_{n,i}$ might be complex valued
in real computations. Hence, at each iteration we apply the following
truncation to the coefficient $c_{n,i}\left( \mathbf{x}\right) $ found by (%
\ref{4.35}) 
\begin{equation}
\widetilde{c}_{n,i}(\mathbf{x}):=%
\begin{cases}
\max \left( |c_{n,i}(\mathbf{x})|,1\right) , & \mathbf{x}\in \Omega
_{T}\times (-0.75,1), \\ 
1, & \text{elsewhere}.%
\end{cases}
\label{8.2}
\end{equation}%
Next, the function $\widetilde{c}_{n,i}(\mathbf{x})$ is smoothed by \texttt{%
smooth3} in MATLAB with the Gaussian option. As a result, a new function $%
c_{n,i}(\mathbf{x})$ is obtained. We do not change notation here for
brevity. Then, the so obtained function $c_{n,i}(\mathbf{x})$ is used for
solving the Lippmann-Schwinger equation in the algorithm~of Section \ref%
{sec:5.2}.

We solve the Lippmann-Schwinger integral equations (\ref{8.8}) using a
spectral method developed in~\cite{Lechl2014, Vaini2000}. This method relies
on the periodization technique introduced by Vainikko, which enables the use
of the fast Fourier transform in the approximation scheme. For solving the
integral equation in our numerical implementation, the domain $\Omega $
defined in (\ref{8.5}) is discretized with the uniform mesh size 0.067 in
three dimensions. The code for this numerical solver was done in MATLAB. The
boundary value problem~(\ref{4.29})--(\ref{4.30}) is solved by the finite
element method. We use tetrahedral mesh with a local adaptive refinement for
the domain $\Omega $. We coded the numerical solver using FreeFem++~\cite%
{Hecht2012}, which is a standard software designed with a focus on solving
partial differential equations using finite element methods. We refer to 
\emph{www.freefem.org} for more information about FreeFem++. All of the
figures that are displayed below in this paper were done by visualization
tools in MATLAB.


\subsection{The first tail function $V_{0}( \mathbf{x}) $}

\label{sec:5.3}

It is clear from the algorithm (at $n=1$) that we use only the gradient $%
\nabla V_{0}$, instead of the first tail function $V_{0}$. Since we set $%
\Delta V_{0}=0$ in Section~\ref{sect:4.3}, we find the derivatives of $V_{0}$
by solving the following problem instead of the problem (\ref{1060}) 
\begin{align}
& \Delta (\partial _{x_{j}}V_{0})=0,\quad \text{in }\Omega ,  \label{1033} \\
& \partial _{x_{j}}V_{0}=\frac{\partial _{x_{j}}u(x,\overline{k})}{u(x,%
\overline{k})},\quad \text{on }\partial \Omega ,\quad j=1,2,3.  \label{103}
\end{align}%
Here, boundary condition \eqref{103} is from \eqref{4.3}. Solving %
\eqref{1033}--\eqref{103} also helps us avoid the error associated with the
numerical differentiation of the first tail function $V_{0}(x)$

For the case where only the backscatter data are given on $\Gamma $, we
simply complete the data on the other parts of the boundary by the incident
plane wave $e^{ikx_{3}}$, see~\eqref{completion}. In doing so, we
approximately assume that $u\left( \mathbf{x},k\right) =e^{ikx_{3}}$  for
points $\mathbf{x}$ located in the domain $\Omega $ near $\partial \Omega
\setminus \Gamma$. Therefore, this completion leads to the boundary
condition for $\partial _{x_{j}}V_{0}$ on $\partial \Omega \setminus \Gamma $
as 
\begin{align}
& \partial _{x_{j}}V_{0}=0,\quad \text{on }\partial \Omega \setminus \Gamma
,\quad j=1,2,  \label{104} \\
& \partial _{x_{3}}V_{0}=i\overline{k},\quad \text{on }\left\{
x_{3}=4.25\right\} \cap \partial \Omega .  \label{105}
\end{align}%
%
%
%
It is seen from condition (\ref{103}) that we need the boundary data for
functions $\partial _{x_{j}}u(\mathbf{x},\overline{k})$, $x\in \Gamma ,$ $%
j=1,2,3$. The derivative with respect to $x_{3}$ for $x\in \Gamma $ is
calculated as 
\begin{equation*}
\partial _{x_{3}}u(\mathbf{x},\overline{k})=\partial _{x_{3}}\tilde{g}(%
\mathbf{x},\overline{k})+i\overline{k}\exp (i\overline{k}x_{3}),
\end{equation*}%
where $\tilde{g}(\mathbf{x},\overline{k})$ is the propagated backscatter
field. The function $\partial _{x_{3}}\tilde{g}(\mathbf{x},\overline{k})$
was found by propagating the backscatter field to the two nearby planes $%
P_{p}=\{x_{3}=-0.75\}$ and $P_{p_{\varepsilon }}=\{x_{3}=-0.75+\varepsilon
\} $, subtracting the results from each other and dividing by $\varepsilon $%
, where we took $\varepsilon =0.1$. Derivatives $\partial _{x_{j}}u(\mathbf{x%
},\overline{k})$, $j=1,2$, for $\mathbf{x}\in \Gamma $ were calculated using
FreeFem++. The propagated data looked very smooth. Again, we have not
observed any instabilities.

\subsection{The stopping criteria and the choice of the final result}

\label{sec:8.4}

We address in this section the rules for stopping the iterations and the
choice of the final result for our algorithm. These rules are guided by the
convergence estimate (\ref{5.161}) of Theorem \ref{thm main} and also by the
trial-and-error testing. We recall that estimate (\ref{5.161}) only claims
that functions $c_{n,i}$ are located in a sufficiently small neighborhood of
the true solution if the number of iterations is not large. However, this
theorem does not claim that these functions tend to the exact solution, also
see Theorem 2.9.4 in~\cite{BeilinaKlibanovBook} and Theorem 5.1 in~\cite{BK}
for similar results. Therefore, our stopping criteria are derived
computationally. We note that the same stopping rules have been used in~\cite%
{Exp1} for the experimental data. For the convenience of the reader we
present these rules, which are taken from~\cite{Exp1}.

We have the stopping criterion for the inner iterations and another one for
the outer iterations. Denote by $e_{n,i}$ the relative error between the two
computed coefficients corresponding by two consecutive inner iterations of
the $n-$th outer iteration 
\begin{equation}
e_{n,i}=\frac{\Vert c_{n,i}-c_{n,i-1}\Vert _{L^{2}(\Omega )}}{\Vert
c_{n,i-1}\Vert _{L^{2}(\Omega )}},\quad i=2,3,\dots  \label{eq:stop}
\end{equation}%
%
%
%
%
%
%
%
%
We consider the $n-$th and $(n+1)-$th outer iterations which contains $I_{1}$
and $I_{2}$ inner iterations respectively. The sequence of relative errors
associated to these two outer iterations is defined by 
\begin{equation}
e_{n,2},\dots ,e_{n,I_{1}},\widetilde{e}_{n+1,1},e_{n+1,2},\dots
,e_{n+1,I_{2}},  \label{eq:sequence}
\end{equation}%
where 
\begin{equation*}
\widetilde{e}_{n+1,1}=\frac{\Vert c_{n+1,1}-c_{n,I_{1}}\Vert _{L^{2}(\Omega
)}}{\Vert c_{n,I_{1}}\Vert _{L^{2}(\Omega )}}.
\end{equation*}

The inner iterations with respect to $i$ of the $n-$th outer iteration in
the above algorithm are stopped when either $e_{n,2}<10^{-6}$ or $i=3$. We
have observed in our numerical experiments that the reconstruction results
are essentially the same when we use either 3 or 5 for the maximal number of
inner iterations.

Concerning the outer iterations with respect to $n$, it can be seen from the
stopping rule for the inner iterations that each outer iteration consists of
at least two and at most three inner iterations. Equivalently, the error
sequence~\eqref{eq:sequence} has at least three and at most five elements.
We stop the outer iterations if there are two consecutive outer iterations
for which their error sequence~\eqref{eq:sequence} has three consecutive
elements less than or equal to $5\times 10^{-4}$. Again, we have no rigorous
justification for these stopping rules; they rely on the content of the
convergence theorem~Theorem \ref{thm main} and on trial-and-error testing.
We point out, however, that we do trial-and-error only for one target, which
we call \textquotedblleft reference target". Then we use the same rules for
all other targets.

We choose the final result for $c_{comp}(\mathbf{x})$ by taking the average
of its approximations $c_{n,i}(\mathbf{x})$ corresponding to the relative
errors in~\eqref{eq:sequence} that meet the stopping criterion for outer
iterations. The computed dielectric constant is determined as the maximal
value of the computed $c_{comp}(\mathbf{x})$, 
\begin{equation}
c_{\max }=\max c_{comp}(\mathbf{x}).  \label{8.6}
\end{equation}%
We have observed in our numerical studies that we need no more than five
outer iterations to obtain the final result.

\subsection{Numerical examples with complete data}

\label{sec:8.5}

The target we consider for this example is a cube of the size 0.6 defined by 
\begin{equation*}
D=(-0.3,0.3)^{2}\times (0,0.6)\subset \mathbb{R}^{3}.
\end{equation*}%
We define the coefficient $c(\mathbf{x})$ as 
\begin{equation}
c(\mathbf{x}):=%
\begin{cases}
5, & \mathbf{x}\in D, \\ 
1, & \text{elsewhere.}%
\end{cases}
\label{eq:coefficient1}
\end{equation}%
We assume that the backscatter data $u_{sc}(\mathbf{x},k)$ are given on the
rectangle (\ref{7}), and on this rectangle the function $g(\mathbf{x},k)$ in
(\ref{4}) is replaced with the function $u_{sc}(\mathbf{x},k)$. We also
assume that the function $u(\mathbf{x},k):=g(\mathbf{x},k)$ in (\ref{4}) is
given on the part $\partial \Omega \setminus \Gamma $ of the domain $\Omega
, $ where $\Gamma $ and $\Omega $ are defined in (\ref{6}) and (\ref{8.5})
respectively. First, we use the data propagation and obtain the function $%
u_{sc}(\mathbf{x},k)$ on the rectangle $\Gamma $ this way. Next, we consider 
$\Omega $ as the computational domain with the measurement data given on its
entire boundary $\partial \Omega $.

Figure 2 displays the reconstruction result using our algorithm. It can be
seen that both the maximal value and the location of the coefficient $c(%
\mathbf{x})$ are well reconstructed. More precisely, for the computed
coefficient $c_{comp}(\mathbf{x})$ we have 
\begin{equation}
c_{\max }=4.9,  \label{5}
\end{equation}%
where $c_{\max }$ is defined in (\ref{8.6}). Hence, it follows from (\ref%
{eq:coefficient1}) and (\ref{5}) that the relative error between the
computed maximal value and the exact maximal value of the coefficient $%
c\left( \mathbf{x}\right) $ is only 2$\%$. We visualize the shape and
location of the exact target and the reconstructed one using \textit{%
isosurface} in MATLAB. The isovalue was chosen as 50$\%$ of the maximal
value of $c_{comp}\left( \mathbf{x}\right) $, and this isovalue is applied
to all the other examples in this paper when using isosurface visualization.

\begin{figure}[h!!!]
\centering
\subfloat[Projection of $c(\x)$ on
$\{y=0\}$]{\includegraphics[width=6cm]{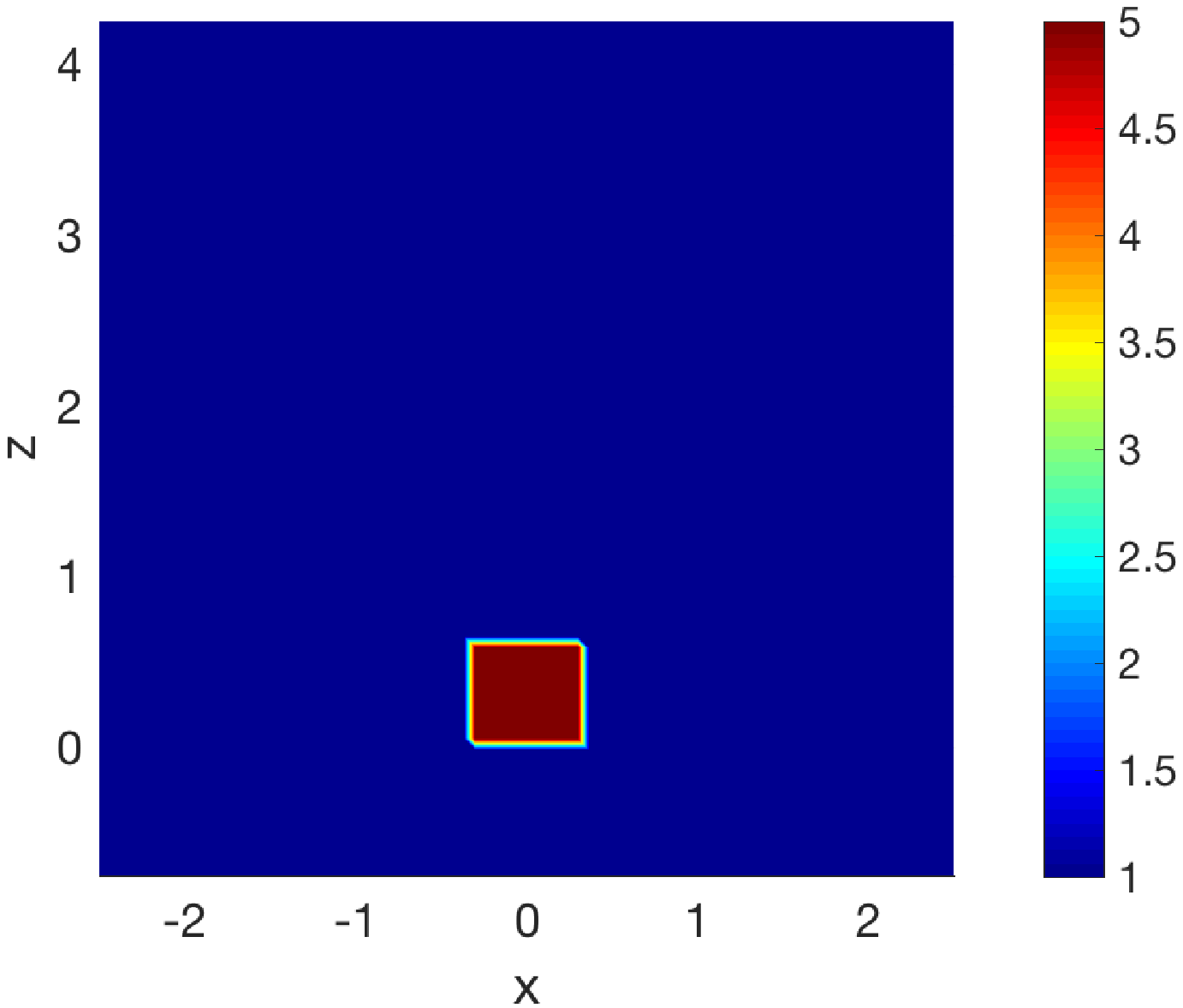}} \hspace{0.5cm} 
\subfloat[Projection of $c_{comp}(\x)$ on
$\{y=0\}$]{\includegraphics[width=6cm]{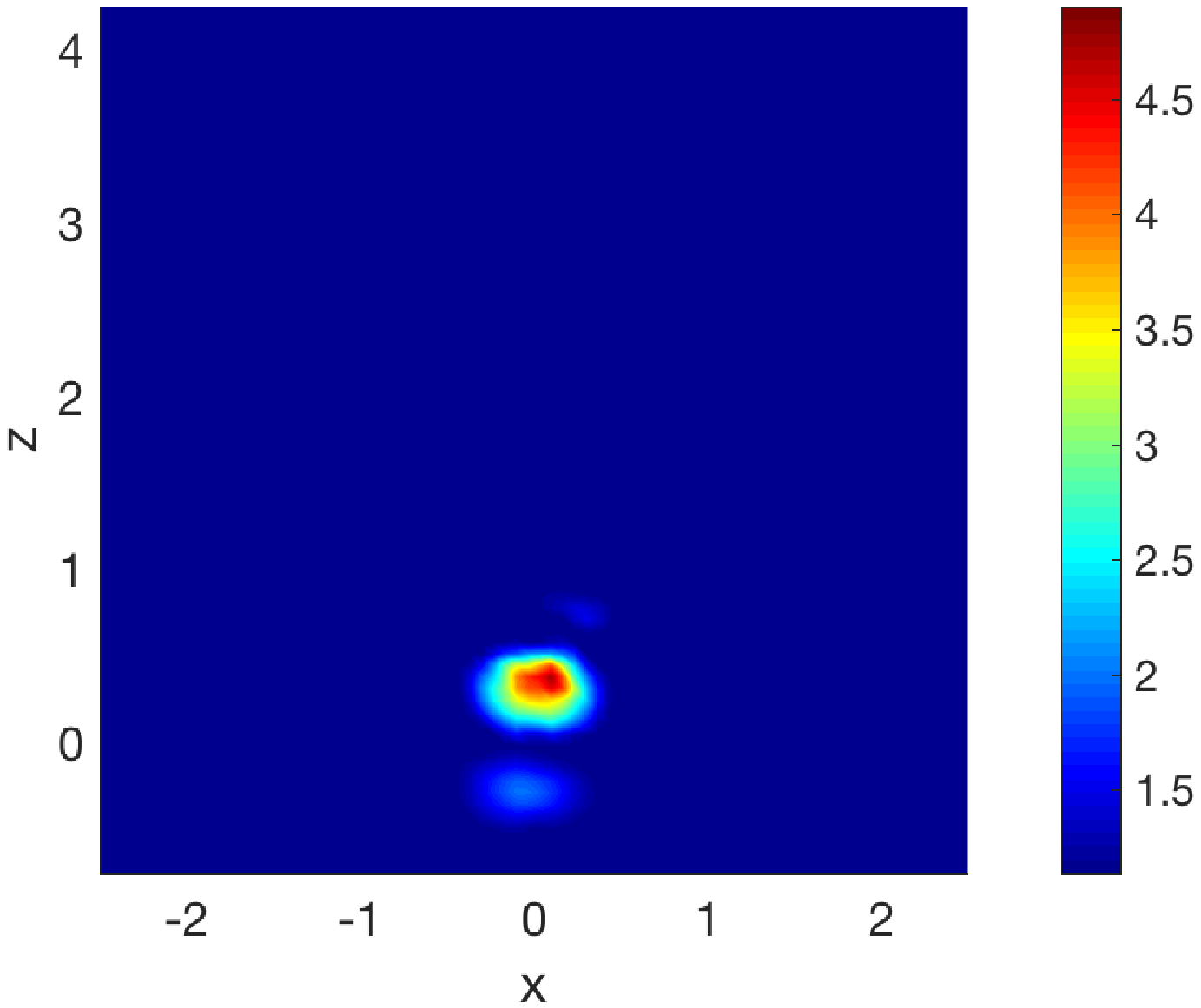}} \vspace{-0.0cm} %
\subfloat[Exact geometry]{\includegraphics[width=6cm]{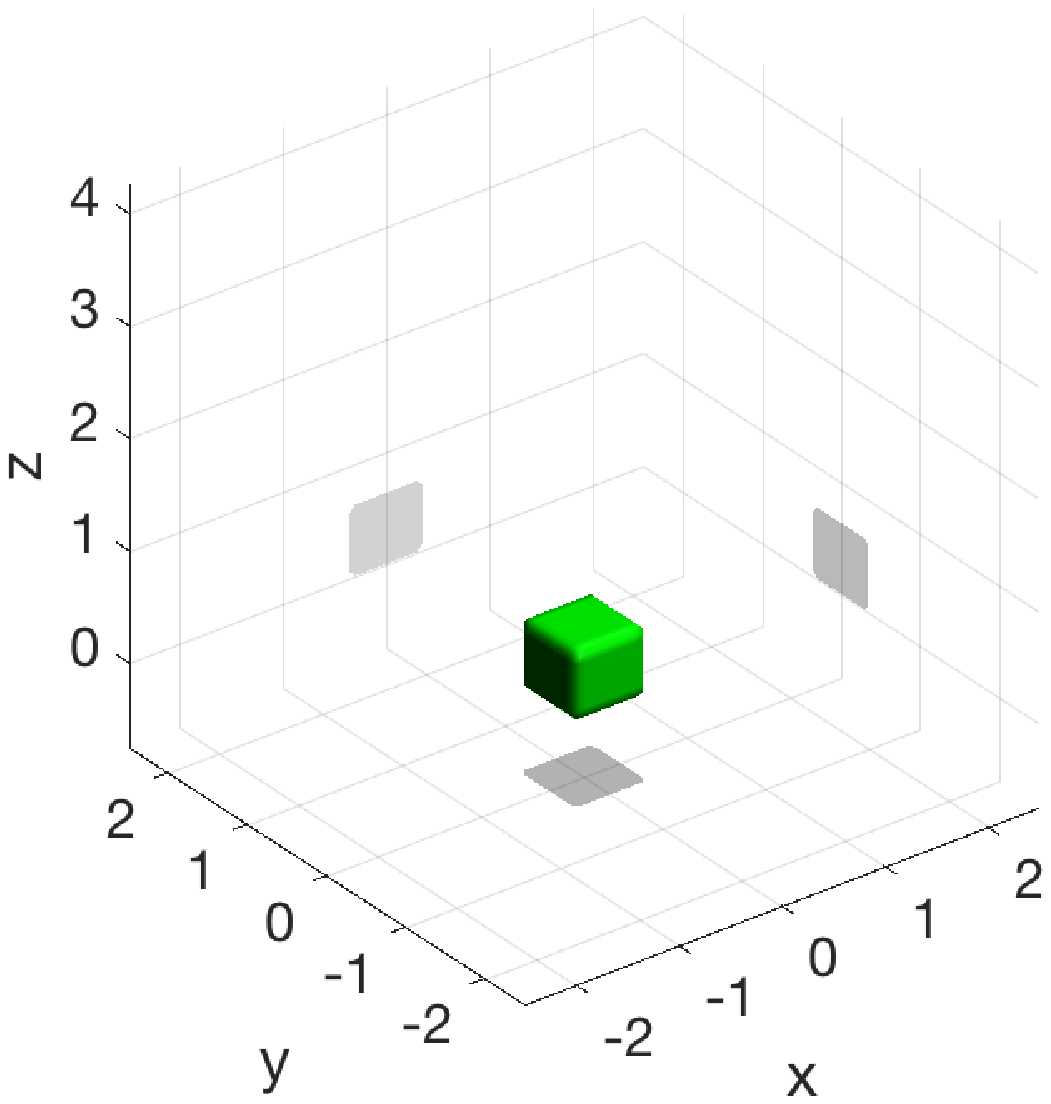}} 
\hspace{0.5cm} 
\subfloat[Reconstructed
geometry]{\includegraphics[width=6cm]{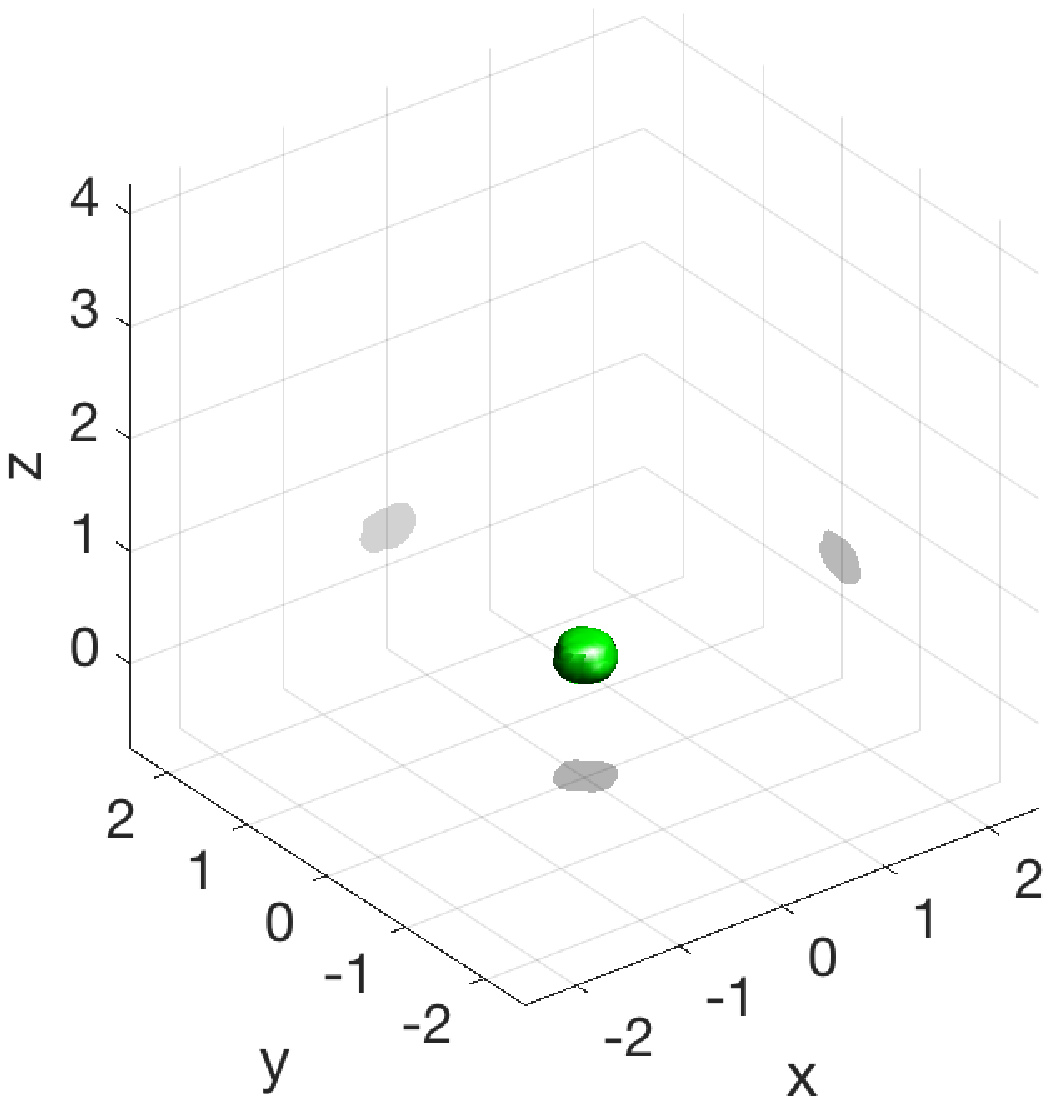}}
\caption{Visualizations of the exact coefficient $c(\mathbf{x})$ in~ 
\eqref{eq:coefficient1} (left) and the reconstructed coefficient $c_{comp}(%
\mathbf{x})$ (right) for the case of complete data with 15$\%$ artificial
noise. The first row is the projection of $c(\mathbf{x})$ and $c_{comp}(%
\mathbf{x})$ on $\{y=0\}$. The last row is a 3D isosurface, with isovalue
2.45, of the exact and reconstructed geometry of the target using MATLAB.}
\label{fi:2}
\end{figure}

\subsection{Numerical examples with backscatter data}

\label{sec:8.6}

We examine in this section the performance of our algorithm for the case of
backscatter data. This case is of our main interest, which is motivated by
our desired application in detecting and identifying of mines and IEDs,
where one typically uses only the backscatter measurement in the stand off
detection. More precisely, by backscatter data we mean that we are only
given the function $u_{sc}\left( \mathbf{x},k\right) $ on the rectangle (\ref%
{7}). Using the data propagation, we can consider the computational domain $%
\Omega $ in~(\ref{8.5}), where we have the backscatter data on the rectangle 
$\Gamma $ defined in (\ref{6}).

Since our method theoretically needs the data on the entire boundary of $%
\Omega $, we complete the missing data on the other parts of $\Omega $ by
the corresponding solution of the forward problem in the homogeneous medium,
where $c(\mathbf{x})=1$. In other words, assuming that $g(\mathbf{x}%
,k)=e^{ikz}+u_{sc}\left( \mathbf{x},k\right) $ is the data given on $\Gamma
, $ we extend it on the entire boundary $\partial \Omega $ as 
\begin{equation}
g(\mathbf{x},k):=%
\begin{cases}
g(\mathbf{x},k), & \mathbf{x}\in \Gamma , \\ 
e^{ikz}, & \mathbf{x}\in \partial \Omega \setminus \Gamma .%
\end{cases}
\label{completion}
\end{equation}%
We remark that data completion methods are widely used for inverse problems
with incomplete data. The data completion~\eqref{completion} is a heuristic
technique relying on the successful experiences of \cite%
{Exp2,Exp1,TBKF2,TBKF1} when working with globally convergent numerical
methods for experimental data.

As the first example for the case of backscatter data, we present in Figure~%
\ref{fi:3} the reconstruction results of the coefficient $c(\mathbf{x})$
defined in~\eqref{eq:coefficient1} for this case. It can be easily seen
again that the location is well reconstructed. For the maximal value we have 
\begin{equation*}
c_{\max }=4.65,
\end{equation*}%
which implies that the relative error between the exact value and the
computed value is 7$\%$. \vspace{0cm}Thus, the error has increased from 2\%
for the case of complete data to 7\% for the case of backscatter data only. 
\begin{figure}[h!]
\centering
\subfloat[Projection of $c_{comp}(\x)$ on
$\{y=0.1\}$]{\includegraphics[width=6cm]{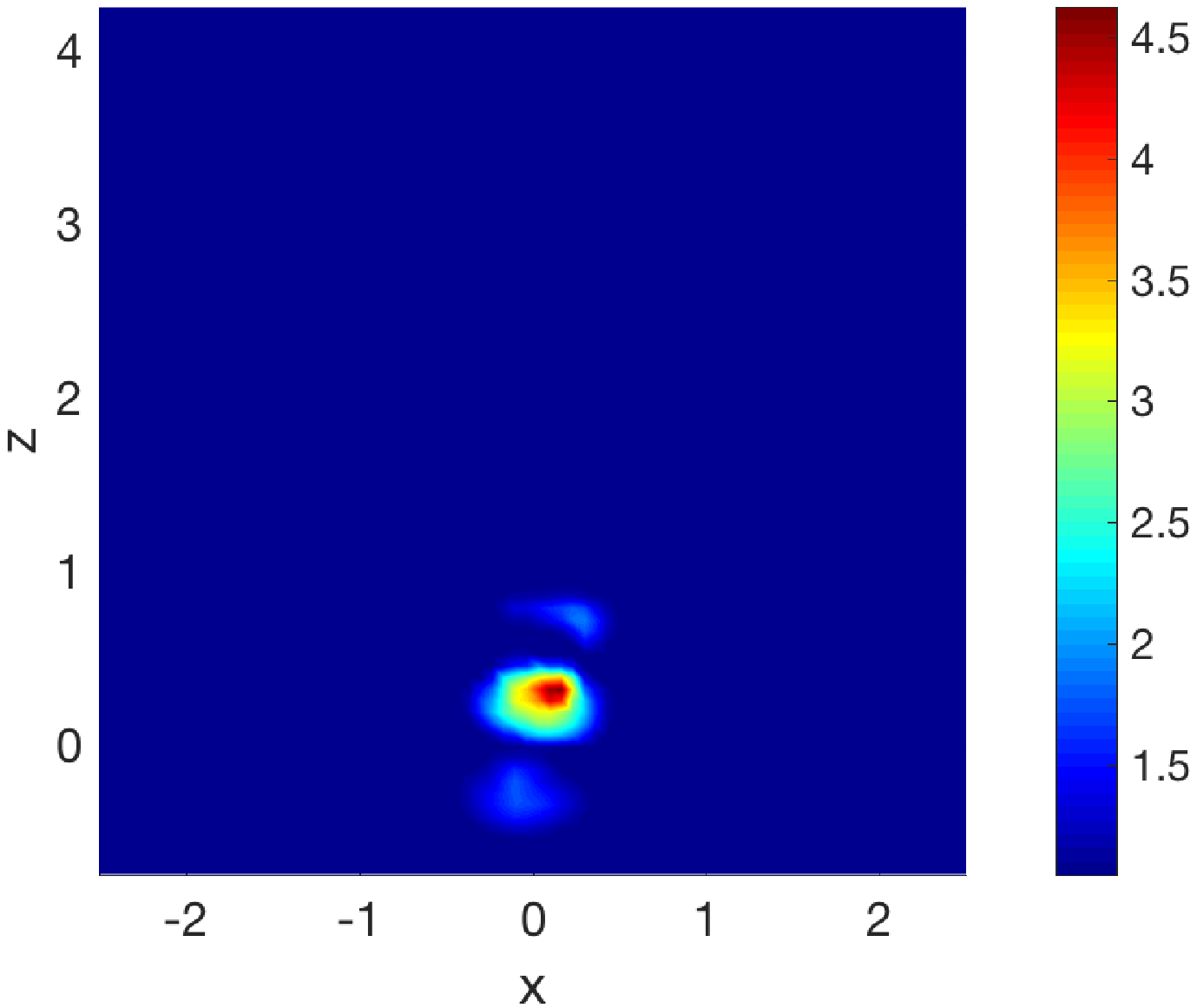}} \hspace{0.5cm} 
\subfloat[Reconstructed geometry
]{\includegraphics[width=6cm]{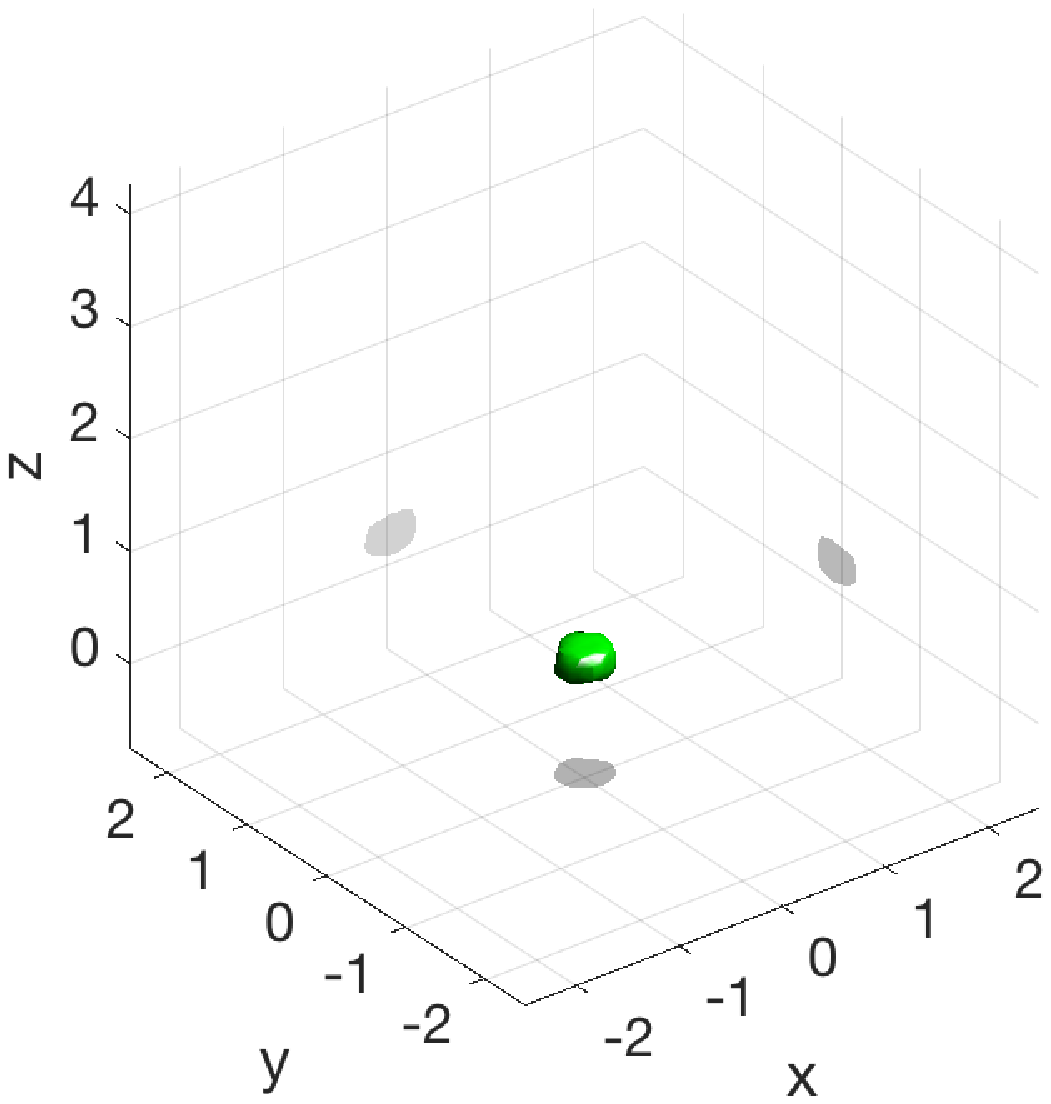}}
\caption{Reconstruction result for the coefficient $c(\mathbf{x})$ in~ 
\eqref{eq:coefficient1} with backscatter data. The left picture is the
projection of $c_{comp}(\mathbf{x})$ on $\{y=0\}$. The right one is the
reconstructed geometry of the target.}
\label{fi:3}
\end{figure}

In the second example we consider the target which consists of two similar
but separate cubes.\ The length of the side of each cube is 0.4. The
coefficient $c(\mathbf{x})$ is defined as 
\begin{equation}
c(\mathbf{x}):=%
\begin{cases}
5, & \mathbf{x}\in D_{1}\cup D_{2}, \\ 
1, & \text{elsewhere},%
\end{cases}
\label{eq:coefficient2}
\end{equation}%
\begin{equation}
D_{1}=(-1,-0.6)\times (-0.2,0.2)\times (0,0.4),\quad D_{2}=(0.6,1)\times
(-0.2,0.2)\times (0,0.4).  \label{9}
\end{equation}%
The reconstruction result for this case is presented in Figure~\ref{fi:4}.
The pictures in the latter figure show a good accuracy of reconstructed
locations. Furthermore, we have 
\begin{equation*}
\max_{D_{1}}\{c_{comp}(\mathbf{x})\}=4.80,\quad \max_{D_{2}}\{c_{comp}(%
\mathbf{x})\}=5.15.
\end{equation*}%
This means the relative errors are 4$\%$ and $3\%$ for the computed values
of the inclusions in $D_{1}$ and $D_{2}$ respectively.

\begin{figure}[h!]
\centering
\subfloat[Projection of $c(\x)$ on
$\{y=0\}$]{\includegraphics[width=6cm]{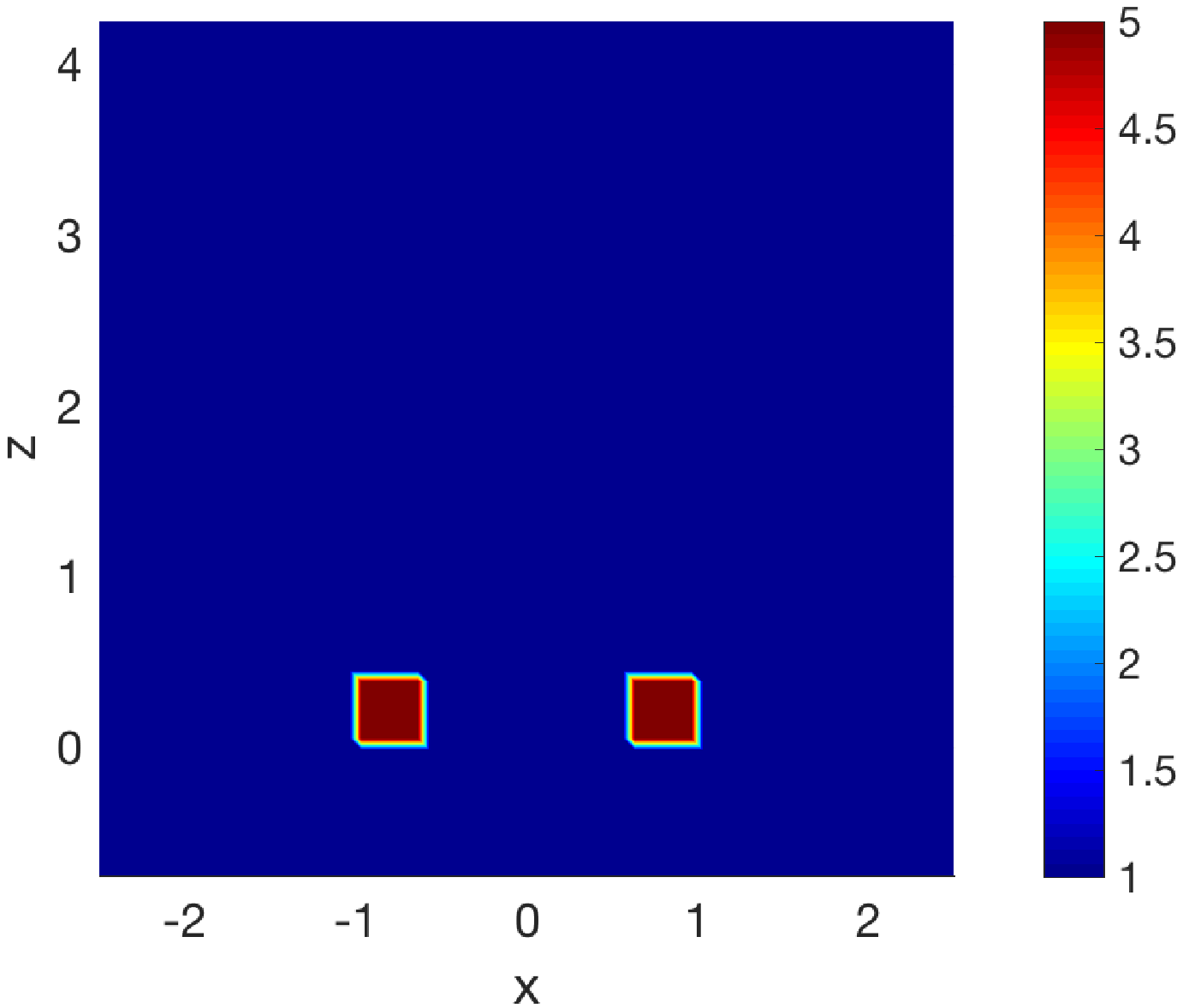}} \hspace{0.5cm} 
\subfloat[Projection of $c_{comp}(\x)$ on
$\{y=0\}$]{\includegraphics[width=6cm]{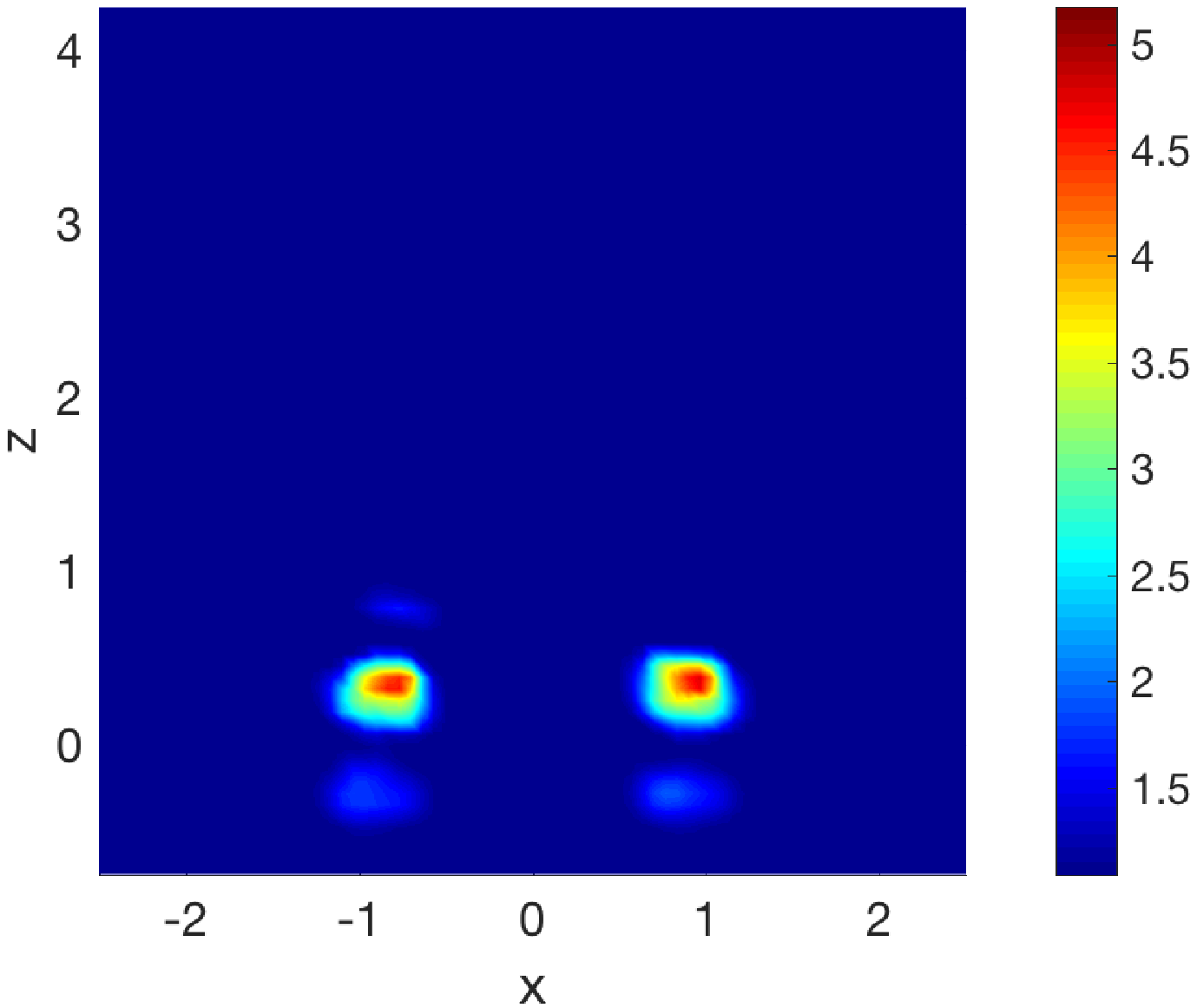}} \\
\subfloat[Exact geometry]{\includegraphics[width=6cm]{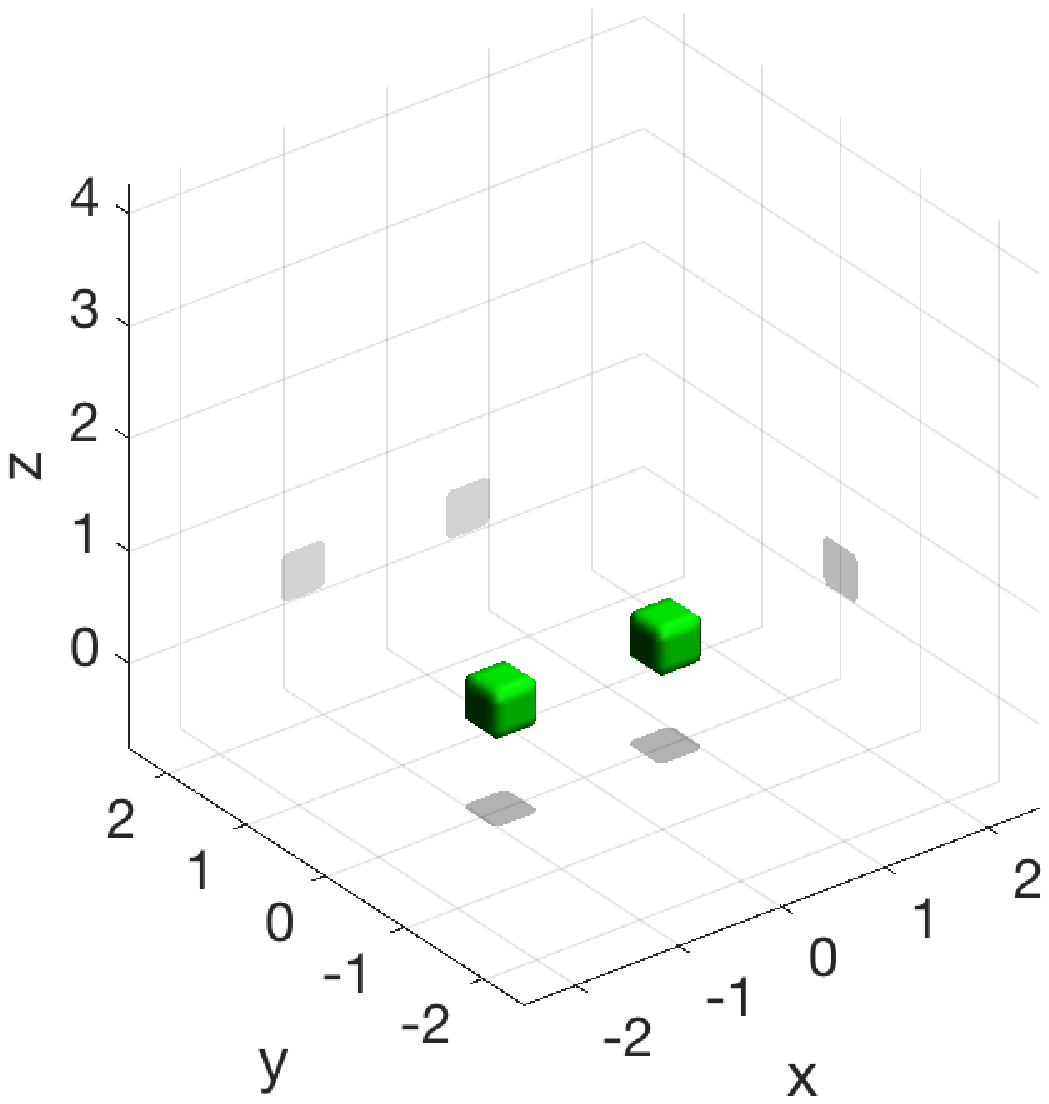}} 
\hspace{0.5cm} 
\subfloat[Reconstructed
geometry]{\includegraphics[width=6cm]{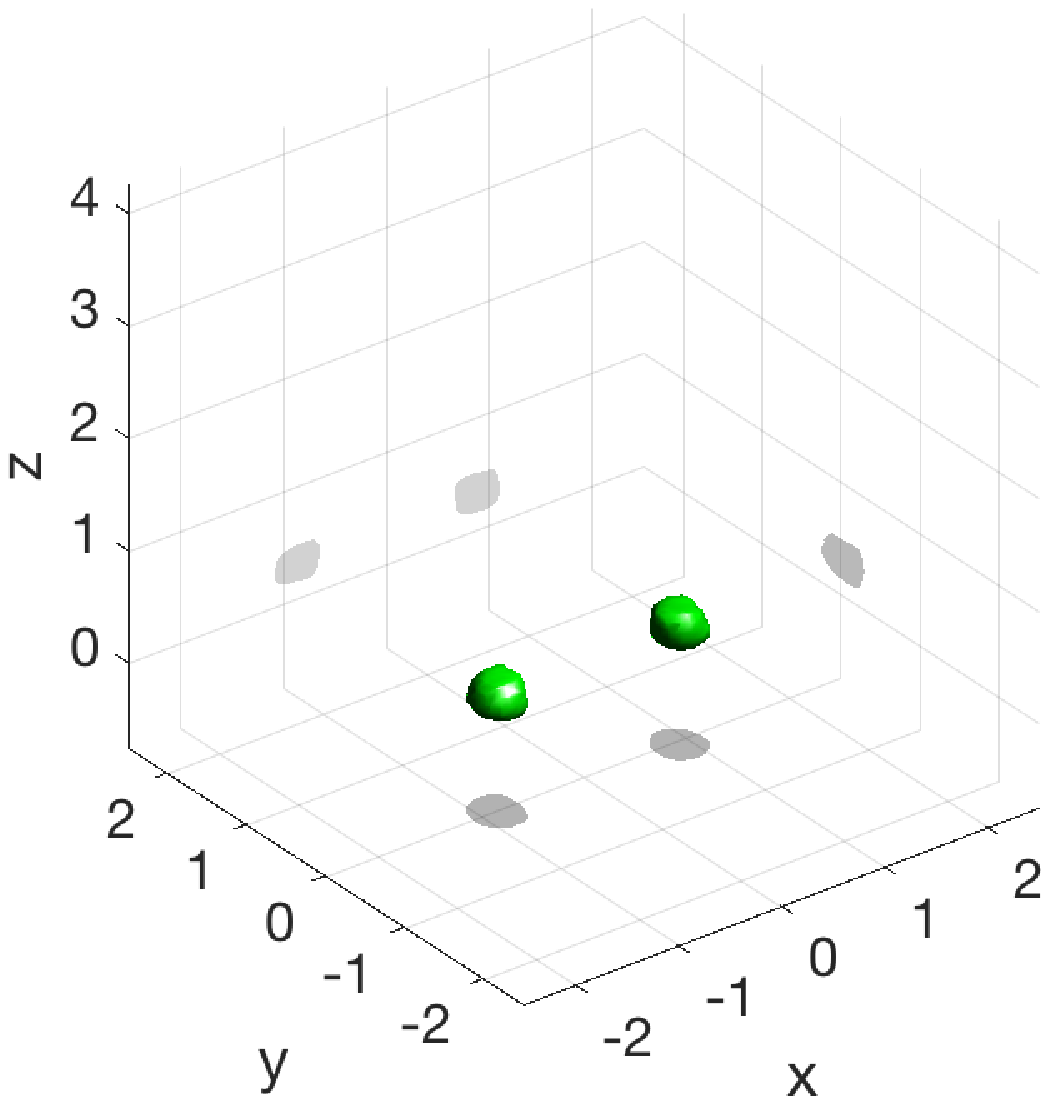}}
\caption{ Visualizations of the exact coefficient $c(\mathbf{x})$ (left) in~ 
\eqref{eq:coefficient2} and the reconstructed coefficient $c_{comp}(\mathbf{x%
})$ (right) for the case of backscatter data. The first row is the
projection of $c(\mathbf{x})$ and $c_{comp}(\mathbf{x})$ on $\{y=0\}$. The
last row is a 3D visualization of the exact and reconstructed geometry of
the target using MATLAB's isosurface. The isovalue is chosen as 50$\%$ of
the maximal value of $c_{comp}(\mathbf{x})$.}
\label{fi:4}
\end{figure}

\subsection{Numerical examples with experimental data}

\label{sec:8.7}

In the last numerical example we address the performance of our algorithm
for experimental multi-frequency data. The data were collected by a
microwave scattering facility at the University of North Carolina at
Charlotte. These measured data are the backscatter signals corresponding to
a single incident plane wave $e^{ikz}$. The setting for the measurement is
similar to that of the above considered backscatter data case.

It is worth mentioning that the raw data are contaminated by a significant
amount of noise, see \cite{Exp1} for a detailed discussion. Furthermore,
standard denoising techniques are inapplicable in this case since the raw
data have a richer content than the computationally simulated data.
Therefore, a heuristic data preprocessing procedure was applied to the raw
data in \cite{Exp1}. The preprocessed data were used then as the input for
our globally convergent algorithm. The aim of the data preprocessing of \cite%
{Exp1} is to make the resulting data look somewhat similar to the
computationally simulated data. We refer to the paper~\cite{Exp1} for all
the details about the experimental setup as well as the data preprocessing.

The target here is a tennis ball of the radius 0.31 (see Figure~\ref{fi:6}%
(a)). The directly measured average dielectric constant of this target at 3
GHz (or equivalently $k=6.28$) is 3.80, with $13\%$ measurement error given
by the standard deviation. This dielectric constant was independently
measured by the physicists Professor Michael A. Fiddy and Mr. Steven Kitchin
from the Department of Physics and Optical Science at the University of
North Carolina Charlotte. Our reconstruction result displayed in Figure~\ref%
{fi:6} indicates that our method accurately reconstructs the location of the
target. The maximal value for the computed coefficient is $c_{\max }=4.00$,
which is 5.3\% error, compared with the direct measurement.

\begin{figure}[h]
\centering
\subfloat[Exact geometry]{\includegraphics[width=6cm]{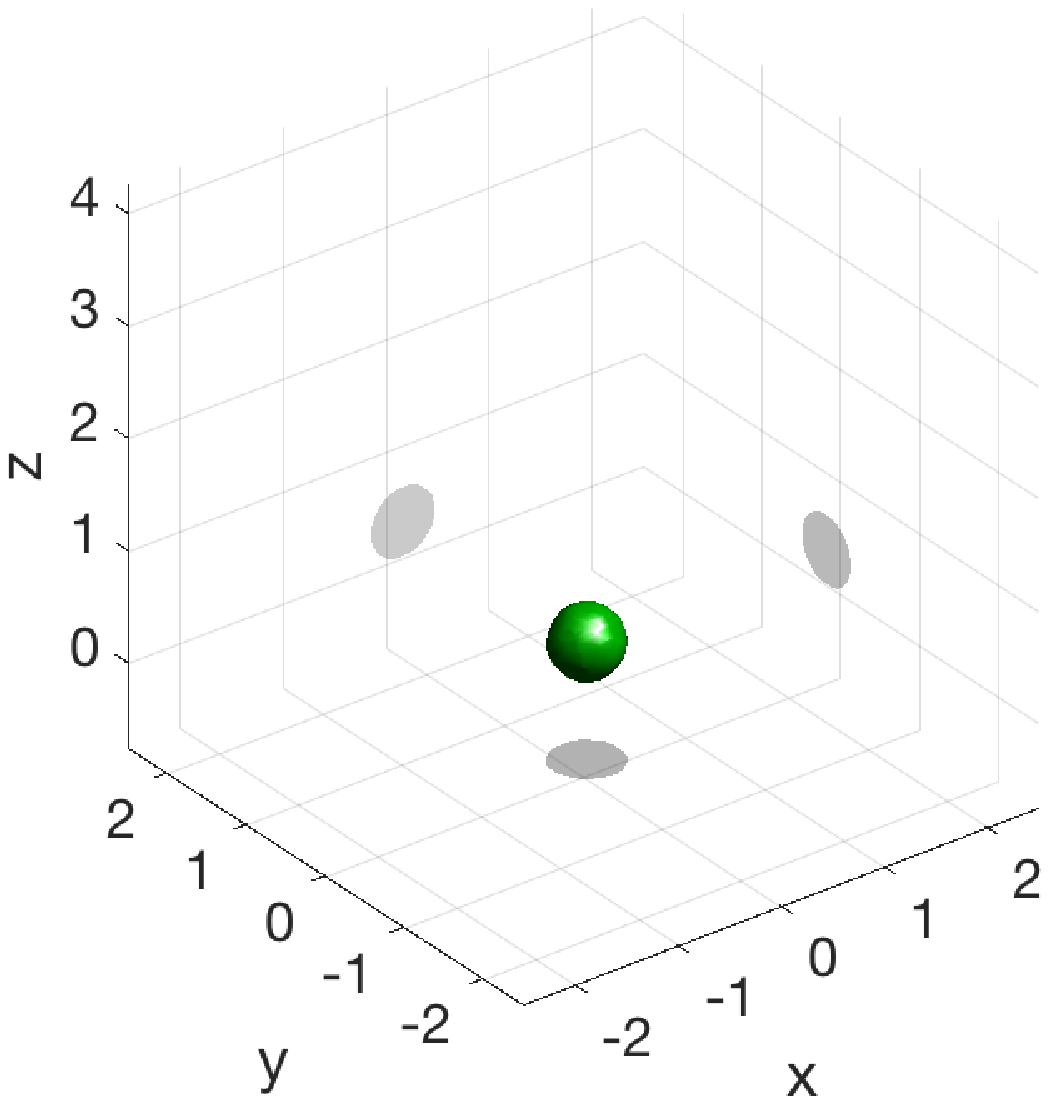}} 
\hspace{0.5cm} 
\subfloat[Reconstructed
geometry]{\includegraphics[width=6cm]{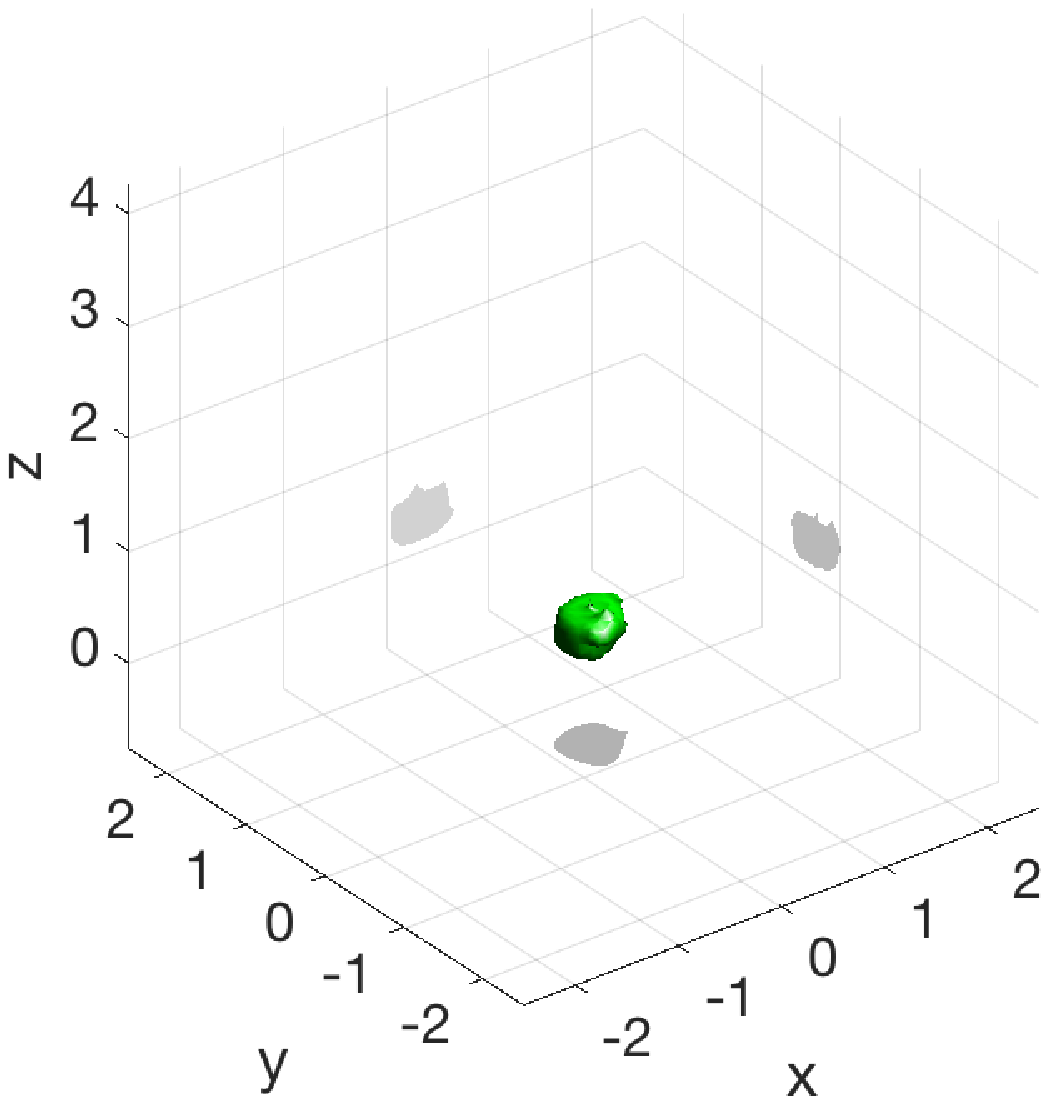}}
\caption{ Visualizations of exact (left) and reconstructed (right) geometry
of the target using the isosurface command in MATLAB.}
\label{fi:6}
\end{figure}

In \cite{Exp1} the above algorithm was tested for five sets of experimental
data. Table 1 presents reconstruction results of the value $c_{\max }$ for
all five cases; \textquotedblleft std.dev" means standard deviation of the
error in the a posteriori measured dielectric constants at the frequency of
3 GHz. One can observe that the relative error of the reconstruction is
between 2\% and 10.3\% in cases 1-4, and it is unclear in the case number 5.

%
%
%

\begin{table}[!h]
\caption{Measured and computed dielectric constants $c$ of the targets}
\label{tab:table2}\centering
\begin{tabular}{|c|c|c|c|}
\hline
Target & Measured $c$\,(std.~dev.) & Computed $c_{\max }$ & Relative error
\\ \hline
A piece of yellow pine & 5.30 (1.6\%) & 5.44 & 2.6\% \\ \hline
A piece of wet wood & 8.48 (4.9\%) & 7.60 & 10.3\% \\ \hline
A geode & 5.44 (1.1\%) & 5.55 & 2.0\% \\ \hline
A tennis ball & 3.80 (13.0\%) & 4.00 & 5.2\% \\ \hline
A baseball & not available & 4.76 & n/a \\ \hline
\end{tabular}%
\end{table}

\section{Summary}

\label{sec:9}

We have analytically developed a new numerical method for solving a
Coefficient Inverse Problem with single measurement data for the 3D
Helmholtz equation. The approximate global convergence of this method is
proved. Our targets of interest mimic antipersonnel mines and improvised
explosive devices. Numerical results for both computationally simulated and
experimental data demonstrate a good accuracy in reconstructed locations and
dielectric constants of targets. It is worth mentioning that, in the case of
experimental data, this good accuracy is achieved regardless on the above
mentioned high level of noise in the data.

\section*{Acknowledgments}

This work was supported by US Army Research Laboratory and US Army Research
Office grant W911NF-15-1-0233 and by the Office of Naval Research grant
N00014-15-1-2330. The authors are grateful to Professor Michael A. Fiddy and
Mr. Steven Kitchin for their excellent work on collection of experimental
data.

\providecommand{\bysame}{\leavevmode\hbox to3em{\hrulefill}\thinspace} %
\providecommand{\MR}{\relax\ifhmode\unskip\space\fi MR } 
\providecommand{\MRhref}[2]{  \href{http://www.ams.org/mathscinet-getitem?mr=#1}{#2}
} \providecommand{\href}[2]{#2}


\end{document}